\documentclass[11pt,twoside]{article}
\usepackage{amssymb}
\usepackage{latexsym}
\usepackage{graphicx}
\usepackage{amsmath}
\usepackage{hyperref,amsthm}
\usepackage{epsfig}
\usepackage{booktabs}
\usepackage{amsmath}

\numberwithin{equation}{section}
\normalsize \makeatletter
\newtheorem{thm}{Theorem}[section]
\newtheorem{lemma}[thm]{Lemma}
\newtheorem{pro}[thm]{Proposition}

\newtheorem{rem}[thm]{Remark}

\newcommand{\be}{\begin{equation}}
\newcommand{\ee}{\end{equation}}
\newcommand{\bea}{\begin{eqnarray*}}
\newcommand{\eea}{\end{eqnarray*}}

 \usepackage{geometry}

\begin{document}
\date{}
\title{\textbf{The blow-up dynamics for the divergence Schr\"{o}dinger equations with inhomogeneous nonlinearity  \footnote{2020 Mathematics Subiect Classification: 35Q55, 35B44.
}}}
\author{Bowen Zheng\footnote{E-mail: \emph{bwen\_zj1516@126.com} (B. Zheng).},\quad \  Tohru Ozawa \footnote{Corresponding author: \emph{txozawa@waseda.jp} (T. Ozawa).
}
\\\emph{College of Sciences, China Jiliang University},\\ \emph{Hangzhou 310018, P. R. China},\\
\emph{Department of Applied Physics, Waseda University},\\ \emph{Tokyo 169-8555, Japan}} \maketitle

\begin{abstract}
This paper is dedicated to the blow-up solution for the divergence Schr\"{o}dinger equations with inhomogeneous nonlinearity (dINLS for short)
\[i\partial_tu+\nabla\cdot(|x|^b\nabla u)=-|x|^c|u|^pu,\quad\quad u(x,0)=u_0(x),\]
where $2-n<b<2$, $c>b-2$, and $np-2c<(2-b)(p+2)$.
First, for radial blow-up solutions in $W_b^{1,2}$, we prove an upper bound on the blow-up rate for the intercritical dNLS. Moreover, an $L^2$-norm concentration in the mass-critical case is also obtained by giving a compact lemma. Next, we turn to the non-radial case. By establishing two types of Gagliardo-Nirenberg inequalities, we show the existence of finite time blow-up solutions in $\dot{H}^{s_c}\cap \dot{W}^{1,2}_b$, where $\dot{H}^{s_c}=(-\Delta)^{-\frac{s_c}{2}}L^2$, and $\dot{W}_b^{1,2}=|x|^{-\frac{b}{2}}(-\Delta)^{-\frac{1}{2}}L^2$. As an application, we obtain a lower bound for this blow-up rate, generalizing the work of Merle and Rapha\"{e}l [Amer. J. Math. 130(4) (2008), pp. 945-978] for the classical NLS equations to the dINLS setting.
\end{abstract}

\section{Introduction}
In this paper, we consider the Cauchy problem of the following divergence Schr\"{o}dinger equations with inhomogeneous nonlinearity (dINLS)
\begin{equation}\label{a0}
\left\{
\begin{aligned}
 & i\partial_tu+\nabla\cdot(|x|^b\nabla u)=-|x|^c|u|^pu,\\
 &u(x,0)=u_0(x), \quad\quad (x,t)\in\mathbb{R}^n\times\mathbb{R},
\end{aligned}\right.
\end{equation}
where $u=u(x,t)$ is the complex-valued function, the parameters $b$ and $c$ satisfy $2-n<b<2$, $c>b-2$, and $\textbf{p}_c:=np-2c<(2-b)(p+2)$.
Such type of variable coefficient Schr\"{o}dinger equation appears in a variety of physical settings, such as plasma physics and the research of non-equilibrium magnetism (see e.g. \cite{LL}, \cite{PT}-\cite{PTW2}, \cite{ST}, \cite{WW3}, \cite{ZZW} and the references therein).

Let us start by observing that if $u(x, t)$ is a solution to \eqref{a0}, so does $u_\lambda$ given by
\[u_\lambda(x, t)=\lambda^{\frac{2-b+c}{p}} u(\lambda x, \lambda^{2-b}t),\quad \mbox{for}\ \lambda>0.\]
A direct calculation shows that
\begin{equation}\nonumber
\|u_\lambda(\cdot,0)\|_{\dot{H}^{s}}=\lambda^{s+\frac{2-b+c}{p}-\frac{n}{2}}\|u_0\|_{\dot{H}^{s}}.
\end{equation}
This gives the critical Sobolev index
\begin{equation}\label{f71}
s_c:=\frac{n}{2}-\frac{2-b+c}{p},
\end{equation}
which is such that the standard homogeneous Sobolev space $\dot{H}^{s_c}=(-\Delta)^{-\frac{s_c}{2}}L^2$ related to \eqref{a0} leaves the scaling symmetry invariant.

We are interested in the case $0\leq s_c<\frac{2-b}{2}$ in this paper. The case $s_c=0$ is known as the \emph{mass-critical}. Rewriting this last condition in terms of $\textbf{p}_c$, we have
\[\textbf{p}_c=2(2-b).\]
On the other hand, the case $0<s_c<\frac{2-b}{2}$ is known as the \emph{mass-supercritical} and \emph{energy-subcritical} (or just \emph{intercritical}). We can also reformulate this condition in terms of $\textbf{p}_c$ as
\[2(2-b)<\textbf{p}_c<(2-b)(p+2),\]
or equivalently, $p_{b, \ast}<p<p_{b}^\ast,$ where $p_{b, \ast}:=\frac{2(2-b+c)}{n}$ and
\begin{equation}p_{b}^\ast:=
\left\{
\begin{aligned}
 & \infty, \quad\quad\quad\quad\quad\ n\leq2-b,\\\nonumber
 &\frac{2(2-b+c)}{n-2+b}, \quad n>2-b.
\end{aligned}\right.
\end{equation}

The main purpose of the present paper is to study long time dynamics of the radial and non-radial solutions to the dINLS \eqref{a0}. As is standard practice, we study the dINLS via its integral version
\[u(t)=e^{it\mathcal{A}_b}u_0+i\int_0^te^{i(t-s)\mathcal{A}_b}(x|^c|u(s)|^pu(s))ds,\]
where $\mathcal{A}_b$ is the operator defined by $\mathcal{A}_b=\nabla\cdot(|x|^b\nabla)$.
Since the variable coefficient $|x|^b$ with $b<0$ can be singular at $x=0$, the local well-posedness analysis of the dINLS \eqref{a0} is more subtle and intricate in several aspects, for example the energy space associated with $\mathcal{A}_b$ is no longer usual Sobolev space, and Strichartz estimate arguments fails. To our aim, we take the LWP in $\dot{W}_b^{s,2}$ related to the dINLS \eqref{a0} for $0\leq s\leq1$ as an assumption when necessary and build part of our conditional result upon it, where $\dot{W}_b^{s,2}:=|x|^{-\frac{b}{2}}(-\Delta)^{-\frac{s}{2}}L^2$ is the weighted Sobolev space. This means that we always assume that for any $u_0\in \dot{W}_b^{s,2}$, the dINLS admits a unique local solution $u$ in $C([0, T^\ast);\ \dot{W}_b^{s,2})$ with the maximal lifetime $0<T^\ast<+\infty$. 
If
\[\lim_{t\rightarrow T^\ast}\|u(t)\|_{\dot{W}_b^{s,2}}=\infty,\]
we call such $u$ is the finite time blow-up solution when $T^\ast<+\infty$.
In addition, the solutions of dINLS formally conserve their mass and energy (see e.g. \cite{PTW2}), 
\begin{equation}\label{d26}
M(u(t)):=\|u\|^2_{2}=M(u_0)
\end{equation}
and
\begin{equation}\label{c10}
E(u(t)):=\frac{1}{2}\|\nabla u\|^2_{b,2}-\frac{1}{p+2}\|u\|^{p+2}_{c,p+2}=E(u_0).
\end{equation}

The dINLS model \eqref{a0} can be interpreted as an extension to the following classical inhomogeneous nonlinear Schr\"{o}dinger equations (INLS)
\begin{equation}\label{d4}
i\partial_tu+\Delta u=-|x|^c|u|^pu
\end{equation}
or even the well-known NLS equation (case $b=c=0$ in \eqref{a0}), which can model nonlinear optical systems with spatially dependent interactions (see \cite{BPVT}).
We recall the known results of blow-up solutions for the INLS equation. Genoud and Stuart \cite{GS} firstly established the local well-posedness in $H^1$. Farah \cite{LGF} proved the existence of blow-up solutions for $|x|u_0\in L^2$. Dinh \cite{DVD} proved the existence of blow-up solutions in radial case without the restriction $|x|u_0\in L^2$. This blow-up result was also extended to general non-radial case by Ardila and Cardoso \cite{AC}, Bai and Li \cite{BL} recently.
We would also like to mention that, the more delicate blow-up dynamics are extensively studied for the INLS equation. Cardoso and Farah \cite{CF} proved an upper bound for the blow-up rate for radial initial data in $H^1$, following the ideas in the remarkable work \cite{MRS} by Merle, Rapha\"{e}l and Szeftel for the classical NLS equation.
Using a profile decomposition technique, Campos and Cardoso \cite{CC} studied the critical $L^2$-norm concentration phenomenon for the INLS equation, (see also \cite{AB}, \cite{VDD}, \cite{G} for various dispersive models).
Some other issues, such as minimal mass blow-up, sharp blow-up rates, etc, related to the blow-up dynamics of INLS had been addressed in \cite{BCD}, \cite{CG2}, \cite{GF2}, \cite{RS} and the references therein. However, to the best of our knowledge, few results are known for the existence of blow-up solutions and the quantitative descriptions of the blow-up dynamics in our divergence case \eqref{a0} with $b\neq0$.

Our first goal of this paper is to investigate the dynamical behavior of the radial blow-up solutions to \eqref{a0} in $W^{1,2}_b$.
As we see, the presence of variable coefficients $|x|^b$ and $|x|^c$ in \eqref{a0} introduce several challenging technical difficulties in the study of this problem.
In the paper \cite{ZOZ}, the authors have proved the existence of finite time blow-up solutions to the dINLS in $W^{1,2}_b$ . The results in \cite{ZOZ} show that, in the intercritical case, radial initial data $u_0\in W^{1,2}_b$ with negative energy will cause finite time blow-up. Else if the initial energy is non-negative and below the ground state, then the finite time blow-up follows from a localized virial estimate and the following radial Gagliardo-Nirenberg inequality
\begin{equation}\label{d21}
\int|x|^c|u|^{p+2}dx\leq C_{GN}\|\nabla u\|_{b,2}^{\frac{\textbf{p}_c}{2-b}}\|u\|_{2}^{\frac{(2-b)(p+2)-\textbf{p}_c}{2-b}}
\end{equation}
for $2-n<b\leq0$, $c\geq b-2$, where the best constant $C_{GN}$ is given by
\begin{equation}\nonumber
C_{GN}=\frac{(2-b)(p+2)\textbf{p}_c^{-\frac{\textbf{p}_c}{2(2-b)}}}{\left((2-b)(p+2)-\textbf{p}_c\right)^{1-\frac{\textbf{p}_c}{2(2-b)}}\|Q\|_2^p}
\end{equation}
and $Q$ is the ground solution to the following elliptic equation
\begin{equation}\label{d27}
\nabla\cdot(|x|^b\nabla Q)+|x|^c|Q|^pQ-Q=0.
\end{equation}
Although the energy threshold gives a sufficient condition on the blow-up, no general quantitative information such as the shape of blow-up solution or the blow-up time rate is known. Therefore, a natural question arising is whether there exists a blow-up solution for the dINLS when the initial data is non-radial. Moreover, what are dynamical properties of those blow-up solutions ?

Inspired by this problem, we first investigate the intercritical dINLS with radial $u_0\in W^{1,2}_b$ and prove a space-time upper bound on the blow-up rate of this solution.
\begin{thm}\label{thm1.3}
For $2-2n\leq b<0$, $c\geq b-2,$ $2(2-b)<\textbf{p}_c<(2-b)(p+2)$ and $p<4$, let $u_0\in W_b^{1,2}$ be radial and assume that
the corresponding solution $u\in C([0,T^\ast);\ W_b^{1,2})$ to \eqref{a0} blows up in finite time. Then the following space-time upper bound
\begin{equation}\label{c23}
\int_t^{T^\ast}(T^\ast-\tau)\|\nabla u(\tau)\|_{b,2}^2d\tau\leq C_{u_0}(T^\ast-t)^{\frac{2\beta}{\beta+1}}
\end{equation}
holds for $\beta=\frac{(4-p)(2-b)}{2\textbf{p}_c-(2-b)p}$, where $C_{u_0}$ is a positive constant depending on $b, \textbf{p}_c, u_0$.
\end{thm}
This theorem follows by an application of the localized virial estimate satisfied by the solutions to \eqref{a0}. It is worth mentioning that Theorem \ref{thm1.3} only holds for $p<4$, which appears in view of our nonlinear estimates \eqref{f91}, see Section 3 below.
We note also that the bound \eqref{c23} implies
\[\liminf_{t\rightarrow T^\ast}(T^\ast-t)^{\frac{1}{1+\beta}}\|\nabla u(t)\|_{b,2}<+\infty,\]
which is an upper bound on the blow-up rate along a sequence of times.

Now we restrict our attention to the mass-critical case, that is, $p=p_{b,\ast}$.
For the dINLS \eqref{a0}, using the energy conservation \eqref{c10} and Gagliardo-Nirenberg inequality \eqref{d21}, we find that
\[E(u_0)\geq \frac{1}{2}\|\nabla u\|_{b,2}^2\left(1-\frac{\|u\|^{p_{b,\ast}}_2}{\|Q\|^{p_{b,\ast}}_2}\right).\]
Thus, if there exists a solution $u$ that blows up in finite time $T^\ast$ satisfying $\|u_0\|_2\geq\|Q\|_2$, then we may have
\[\sup_{t\in [0, T^\ast)}\|u(t)\|_2\geq\|Q\|_2.\]
This suggests us to investigate the occurrence of the $L^2$-norm concentration for finite time blow-up solutions to the dINLS equation.

Unlike the classical NLS equation, the dINLS \eqref{a0} is not invariant under space translation. The profile decomposition in Hmidi and Kerraani \cite{HK} that is always used to deal with the $L^2$-norm concentration phenomenon will be unavailable for the divergence case. This fact leads to some additional technical difficulties for obtaining the $L^2$-norm concentration result for \eqref{a0}. In this paper, we will take the advantage of the compactness embedding result presented in Lemma \ref{lem1}, and  prove that the $L^2$-norm concentration for the mass-critical dINLS occurs at the origin for finite time blow-up solutions. This result is formulated below.

\begin{thm}\label{thm1.1}
Let $2-n<b\leq0$ and $b-2<c<b-2+2n$. Let $u$ be a radial $W_b^{1,2}$-solution to \eqref{a0} in the mass-critical case
and assume that it blows up in finite time $T^\ast>0$ satisfying
\begin{equation}\label{c8}
\|\nabla u(t)\|_{b,2}\geq\frac{C}{\sqrt{T^\ast-t}},\quad\quad 0\leq t<T^\ast
\end{equation}
with some universal constant $C>0$ and $t$ close enough to $T^\ast$.
If $\lambda\in C([0, T^\ast))$ is a positive function satisfying $\lambda(t)\|\nabla u(t)\|_{b,2}^{\frac{2}{2-b}}\rightarrow+\infty$ as $t\rightarrow T^\ast$, then

(1) (\textbf{$L^2$-norm concentration})
\begin{equation}\label{c7}
\liminf_{t\rightarrow T^\ast}\int_{|x|\leq\lambda(t)}|u|^2dx\geq\|Q\|_2^2,
\end{equation}
where $Q$ is the  solution to the elliptic equation \eqref{d27}.

(2) there is no sequence $\{t_n\}$ such that $t_n\rightarrow T^\ast$ and $u(t_n)$ converges strongly in $L^2$ as $n\rightarrow\infty$.
\end{thm}

\begin{rem}
(1) It is possible to deduce the existence of $\lambda(t)$ such that $\lambda(t)\rightarrow0$, as $t\rightarrow T^\ast$, from the lower bound \eqref{c8} on the blow-up rate. For example, $\lambda(t)=(T^\ast-t)^\alpha$ with $0<\alpha<\frac{1}{2-b}$ satisfies the assumption of Theorem \ref{thm1.1}.

(2) In \cite{CC}, Campos and Cardoso showed a similar result for the INLS equation in the mass-critical case $(p=\frac{2(2+c)}{n})$. Theorem \ref{thm1.1} generalizes this result to the case $2-n<b\leq0$ for the dINLS setting.
\end{rem}

Next, we are devoted to the blow-up problem for the dINLS \eqref{a0} with low regularity initial data in $\dot{H}^{s_c}\cap \dot{W}_b^{1,2}$.
To our knowledge, Merle and Rapha\"{e}l \cite{MR} established a sufficient condition for the existence of radial finite time blow-up solutions to the NLS equation in $\dot{H}^{s_0}\cap \dot{H}^1$ with $s_0=\frac{n}{2}-\frac{2}{p}$, and provided a lower bound for the blow-up rate of its $\dot{H}^{s_0}$ norm.
In the recent work \cite{CF}, Cardoso and Farah assumed similar conditions with radial negative energy initial data and extended this results to the INLS setting. However, their proofs do not apply to the non-radial case. It is worth noting that the problem \eqref{a0} does not fall within the scope of \cite{CF, MR}.  As a matter of fact, the situation becomes much more delicate for our divergence case \eqref{a0}. It is unclear whether there exist blow-up solutions to the dINLS in $\dot{H}^{s_c}\cap \dot{W}_b^{1,2}$.

In the next theorem, we will remove the radiality assumption and present the existence of finite time blow-up solutions to the dINLS \eqref{a0} for non-positive energy initial data in $\dot{H}^{s_c}\cap \dot{W}_b^{1,2}$.
\begin{thm}\label{thm2}
Let $n\geq3,\ 2-n<b\leq0,\ b-2<c<\frac{nbp}{4},\ 2(2-b)<\textbf{p}_c<(2-b)(p+2),\ p<\frac{4}{n}$ and let $\sigma_0$ be defined as
\begin{equation}\label{f64}
\sigma_0=\frac{(n+\gamma)p}{2-b+c}
\end{equation}
with $-n<\gamma\leq0$. If $u_0\in \dot{H}^{s_c}\cap \dot{W}_b^{1,2}$ and $E(u_0)\leq0$, then the corresponding solution $u$ to \eqref{a0} blows up in finite time $0<T^\ast<\infty$.

\end{thm}

We point out that there are two major challenges in the proof of Theorem \ref{thm2}. One challenge is that we the mass concentration law does not work due to the initial data condition $u_0\in \dot{H}^{s_c}\cap \dot{W}_b^{1,2}$. To overcome it, we will follow the approach introduced by Merle and Rapha\"{e}l \cite{MR} and replace the role of the $L^2$ norm by a suitable invariant Morrey-Campanato norm, see the definition in \eqref{f26} below. But since the emergence of unbounded variable coefficients $|x|^b$ and $|x|^c$ in the dINLS case, we mention that the standard Sobolev space theory are not useful for the argument of the existence of blow-up solutions, which is another delicate problem. Thus, we must search for some new tools to deal with the nonlinearity, and a finer analysis is required to control the energy of the variable coefficient presented in the linear part of \eqref{a0}. To do it, one key ingredient in the proof of Theorem \ref{thm2} is the following sharp
Gagliardo-Nirenberg inequality
\begin{equation}\nonumber
\int|x|^c|f|^{p+2}dx\leq \frac{p+2}{2\|\mathcal{Q}\|_{\gamma,\sigma_0}^p}\|\nabla f\|_{b,2}^2\|f\|_{\gamma,\sigma_0}^p,
\end{equation}
where $\mathcal{Q}$ is a solution to the elliptic equation
\begin{equation}\nonumber
\nabla\cdot(|x|^b\nabla \mathcal{Q})+|x|^c|\mathcal{Q}|^{p}\mathcal{Q}-|x|^\gamma|\mathcal{Q}|^{\sigma_0-2}\mathcal{Q}=0
\end{equation}
and $\sigma_0$ is defined as in \eqref{f64}.
We also note that this Gagliardo-Nirenberg inequality guarantees that the quantity $E(u(t))$ is well-defined for functions in $\dot{W}_b^{1,2}\cap L^{\sigma_0}(|x|^\gamma dx)$.
Moreover, we will show that the following Sobolev embedding
\[\dot{H}^{s_c}\hookrightarrow L^{\sigma_0}(|x|^\gamma dx)\]
holds for all  $-n<\gamma\leq0$ and $0<s_c<\frac{2-b}{2}$. In virtue of this interpolation inequality and compactness argument, the analysis of the blow-up solutions to the dINLS \eqref{a0} in $\dot{H}^{s_c}\cap \dot{W}_b^{1,2}$ is now reduced to some crucial uniform estimates, which occupy the main part of Section 5 and strongly benefit from a localized virial estimate adapted to our dINLS setting.

The last result describes the behavior of the $\dot{H}^{s_c}$-norm for any finite time blow-up solution to the dINLS \eqref{a0} with initial data in $\dot{H}^{s_c}\cap \dot{W}_b^{1,2}$. Our result recovers and extends the ones in previous works (\cite{CF}, \cite{MR}).
\begin{thm}\label{thm3}
Let $b,\ c,\ p,\ \textbf{p}_c$ and $\sigma_0$ be as in Theorem \ref{thm2}.
If $u_0\in \dot{H}^{s_c}\cap \dot{W}_b^{1,2}$ is such that the corresponding solution to \eqref{a0} blows up in finite time $0<T^\ast<+\infty$ and satisfies
\begin{equation}\label{f92}
\|\nabla u(t)\|_{b,2}\geq\frac{C}{(T^\ast-t)^{\frac{2-b-2s_c}{2(2-b)}}}
\end{equation}
for some constant $C$ and $t$ close enough to $T^\ast$, then there exists a constant $l=l(n, b, c, p, \gamma)>0$ such that
\begin{equation}\label{f75}
C\|u(t)\|_{\dot{H}^{s_c}}\geq\|u(t)\|_{\gamma, \sigma_0}\geq|\log(T^\ast-t)|^l.
\end{equation}
\end{thm}

The scaling power in the lower bound \eqref{f92} is very natural and it is easily deduced if a local Cauchy theory in $\dot{W}_b^{s,2}$ is available (see \cite{AT} for the INLS case). In fact, let $u_0\in \dot{W}_b^{s,2}$ be such that $T^\ast<\infty$ and $u\in C([0, T^\ast);\ \dot{W}_b^{s,2})$ be the maximal solution to \eqref{a0}. For a fixed $t\in [0, T^\ast)$, we consider the following scaling of $u$ as
\[v^{(t)}(x,\tau)=\lambda(t)^{\frac{2(2-b+c)}{(2-b)p}}u(\lambda^{\frac{2}{2-b}}(t)x, t+\lambda^2(t)\tau)\]
with $\lambda(t)^{\frac{2(s-s_c)-b}{2-b}}\|u\|_{\dot{W}_b^{s,2}}=1$. So we get $\|v^{(t)}(0)\|_{\dot{W}_b^{s,2}}=1.$
From the local theory in $\dot{W}_b^{s,2}$, there exists a $\tau_0>0$, independent of $t$, such that $v^{(t)}$ is defined on $[0, \tau_0]$. Then $t+\lambda^2(t)\tau_0<T^\ast$, which arrives at \eqref{f92} by choosing $s=1$. In our work,
due to the lack of the local Cauchy theory in $\dot{W}^{s,2}_{b}$ for the dINLS model, we include this assumption on the statement of Theorems \ref{thm1.1} and \ref{thm3}.

This paper is structured as follows. In Section 2, we introduce some notations and collect some technical tools which we will need along the proof. Section 3 is devoted to the proof of two dynamical properties of the radial blow-up solutions in $W_b^{1,2}$, including Theorems \ref{thm1.3} and \ref{thm1.1}. In Sections 4 and 5, we establish two types of generalized Gagliardo-Nirenberg inequalities, and use them to deduce the uniform estimates in Propositions \ref{pro4.1} and \ref{pro4.2}, which are useful in the proof of Theorem \ref{thm2}. In Section 6, we give the proof of Theorems \ref{thm2} and \ref{thm3}.

\section{Preliminary}
We start this section by introducing the notation used throughout the paper. We use $C_\alpha$ to denote various positive constants that depend on $\alpha$ and may vary line by line. All the integrals will be understood to be over $\Bbb{R}^n$, except as specified.
We use $L^q(\Bbb{R}^n; \omega(x)dx)$ to denote the weighted Lebesgue space, which is defined via
\[\|u\|_{L^q(\Bbb{R}^n; \omega(x)dx)}=(\int\omega(x)|u|^qdx)^{\frac{1}{q}}.\]
We make the usual modification when $q$ equals to $\infty$. We often abbreviate the space $L^q(\Bbb{R}^n; \omega(x)dx)$ by $L^q(\omega(x)dx)$, the norm $\|\cdot\|_{L^q(\Bbb{R}^n; |x|^adx)}$ by $\|\cdot\|_{a,q}$ in case $\omega(x)=|x|^a$. In particular, we denote $\|\cdot\|_{q}=\|\cdot\|_{0,q}$.

Given $a,\ s\in\Bbb{R},\ 1\leq q<\infty$, we also define the inhomogeneous (resp. homogeneous) weighted Sobolev space $W_a^{s,q}(\Bbb{R}^n)$ (resp. $\dot{W}_{a}^{s,q}(\Bbb{R}^n)$) via the norms
\[\|u\|_{W_{a}^{s,q}}=\|(-\Delta)^{\frac{s}{2}} u\|_{a,q}+\|u\|_{q}\quad\mbox{and}\quad  \|u\|_{\dot{W}_{a}^{s,q}}=\|(-\Delta)^{\frac{s}{2}} u\|_{a,q}.\]
To shorten formulas, we often omit $\Bbb{R}^n$. In particular, if $a=0$, $W^{s, p}=W_0^{s, p}$ and $\dot{W}^{s, p}=\dot{W}_{0}^{s,p},$ where $W^{s, p}$ (resp. $\dot{W}^{s, p}$) are the standard inhomogeneous (resp. homogeneous) Sobolev spaces.
As usual, $H^s=W_0^{s,2}$ and $\dot{H}^s=\dot{W}_{0}^{s,2}.$

\subsection{Weighted Sobolev embedding}
In this section, we mainly prove a compact embedding result in Lemma \ref{lem4} below for general non-radial functions in $\dot{W}_b^{1,2}\cap L^{\sigma_0}(|x|^\gamma dx)$.
To proceed it, we first recall a weighted Sobolev embedding lemma in an open bounded set $\Omega$, which appeared in \cite{GGW}.
\begin{lemma}\label{lem3}
Let $n\geq3$ and $\Omega\subset\Bbb{R}^n$ be an open bounded set. Assume that $b>2-n$ and $b-2<c\leq\frac{nb}{n-2}$. Then the embedding
\[\dot{W}_b^{1,2}(\Omega)\hookrightarrow L^q(\Omega;\ |x|^cdx),\]
is continuous provided $q\in[1, \frac{2n+2c}{n-2+b}]$ and this embedding is compact for $q\in[1, \frac{2n+2c}{n-2+b})$.
\end{lemma}
In the case $\Omega=\Bbb{R}^n$, we also recall the following celebrated Caffarelli-Kohn-Nirenberg inequality \cite{CKN}.
\begin{lemma}\label{lem3.3}
For $b>-n$, there exists $C>0$ such that the inequality
\begin{equation}\label{d5}
(\int_{\Bbb{R}^n}||x|^\alpha f|^{q+2}dx)^{\frac{1}{q+2}}\leq C(\int_{\Bbb{R}^n}||x|^{\frac{b}{2}}\nabla f|^2dx)^{\frac{a}{2}}(\int_{\Bbb{R}^n}||x|^\beta f|^ldx)^{\frac{1-a}{l}}
\end{equation}
holds if and only if the following relation
\begin{equation}\nonumber
\frac{1}{q+2}+\frac{\alpha}{n}=a\left(\frac{1}{2}+\frac{b-2}{2n}\right)+(1-a)\left(\frac{1}{l}+\frac{\beta}{n}\right)
\end{equation}
holds, where $\frac{1}{q+2}+\frac{\alpha}{n}>0,\ \frac{1}{l}+\frac{\beta}{n}>0$, $\alpha\leq(1-a)\beta$ with $0\leq a\leq1$.
\end{lemma}
\begin{rem}\label{rem1}
We note that

1. Taking $l=2$, $\beta=0$ in the inequality \eqref{d5}, we obtain the following Gagliardo-Nirenberg inequality
\begin{equation}\nonumber
\int|x|^c |f|^{p+2}dx\leq C(\int|x|^b|\nabla f|^2dx)^{\frac{\textbf{p}_c}{2(2-b)}}(\int|f|^2dx)^{\frac{(2-b)(p+2)-\textbf{p}_c}{2(2-b)}},
\end{equation}
for $-n<c\leq0$ and $\textbf{p}_c\leq(2-b)(p+2)$, where the version of radial symmetric function \eqref{d21} was established by the authors in \cite{ZOZ}.

2. A special case of \eqref{d5} is the following known Hardy-Sobolev inequality
\begin{equation}\label{c26}
(\int |x|^{-\frac{2nd}{n-2+b+2d}}|f|^{\frac{2n}{n-2+b+2d}}dx)^{\frac{n-2+b+2d}{n}}\leq C_{b,d}\int |x|^b|\nabla f|^2dx,
\end{equation}
which holds for $n\geq3,\ b>2-n$ and $-\frac{b}{2}\leq d\leq \frac{2-b}{2}$. See more details in \cite{WW}.
\end{rem}

Based on the above lemmas, we are now ready to prove the following compact weighted Sobolev embedding for general non-radial functions in $\dot{W}_b^{1,2}\cap L^{\sigma_0}(|x|^\gamma dx)$, which is essential to prove the improved Gagliardo-Nirenberg inequality stated in Proposition \ref{pro1}.
\begin{lemma}\label{lem4}
Let $n\geq3$, $2-n<b<2,\ b-2<c\leq \frac{nb}{n-2}$ and let $\sigma_0,\ \gamma$ be defined as \eqref{f64}. If the function $\omega$ satisfies
\begin{equation}\label{f68}
\limsup_{|x|\rightarrow0}\frac{\omega(x)}{|x|^c}<\infty\quad\quad\mbox{and}\quad\quad \limsup_{|x|\rightarrow\infty}\frac{\omega(x)}{|x|^c}<\infty,
\end{equation}
then the following embedding
\begin{equation}\label{d1}
\dot{W}_b^{1,2}\cap L^{\sigma_0}(|x|^\gamma dx)\hookrightarrow L^{p+2}(\omega(x)dx),
\end{equation}
is compact for all $\textbf{p}_c$ with $2(2-b)<\textbf{p}_c<(2-b)(p+2)$.
\end{lemma}
\begin{proof}
Let $\{f_n\}_{n\in\Bbb{N}}$ be a bounded sequence in $\dot{W}_b^{1,2}\cap L^{\sigma_0}(|x|^\gamma dx)$. Then there exists a function $f\in \dot{W}_b^{1,2}\cap L^{\sigma_0}(|x|^\gamma dx)$ such that, up to subsequence,
\begin{equation}\label{f69}
f_n\rightharpoonup f\quad\quad \mbox{in}\ \ \dot{W}_b^{1,2}\cap L^{\sigma_0}(|x|^\gamma dx)
\end{equation}
as $n\rightarrow\infty$. Define $g_n=f_n-f$. To prove the compactness of the embedding \eqref{d1}, it suffices to show that
\[\lim_{n\rightarrow\infty}\int \omega(x)|g_n|^{p+2}dx=0.\]

To do it, by \eqref{f68}, there exists $R_1>r_1>0$ and some constants $C>0$ such that
\begin{equation}\label{d3}
\omega(x)\leq C|x|^c,\quad\quad\ \mbox{for}\quad 0<|x|\leq r_1
\end{equation}
and
\begin{equation}\label{d2}
\omega(x)\leq C|x|^c,\quad\quad\ \mbox{for}\quad |x|\geq R_1.\quad\quad
\end{equation}
Given any small $\epsilon>0$, we first show that there exists $0<r<r_1$ such that
\begin{equation}\label{d10}
\int_{B_r}\omega(x)|g_n|^{p+2}dx<\frac{\epsilon}{3},
\end{equation}
where $B_\rho=\{x\in\Bbb{R}^n:\ |x|\leq \rho\}$ is a ball with the radius $\rho$.

Indeed, we choose a smooth cut-off function $\psi\in C_0^\infty(\Bbb{R}^n)$ satisfying
\begin{equation}\psi(x)=\nonumber
\left\{
\begin{aligned}
 &1, \quad |x|\leq\frac{r_1}{2},\\
 & 0, \quad|x|\geq r_1,
\end{aligned}\right.\ \quad\mbox{and}\quad\quad |\nabla\psi|\leq C.
\end{equation}
Then for $0<r<\frac{r_1}{2}$, we use \eqref{d3} to get
\begin{eqnarray}\nonumber
\lefteqn{\int_{B_r}\omega(x)|g_n|^{p+2}dx\leq C\int_{B_r}|x|^c|g_n|^{p+2}dx}\\\nonumber
&&\quad\quad\quad\quad\quad\quad\ \leq Cr^{\frac{(2-b)(p+2)-\textbf{p}_c}{2}}\int|x|^{\frac{np-(2-b)(p+2)}{2}}|\psi g_n|^{p+2}dx\\\nonumber
&&\quad\quad\quad\quad\quad\quad\ \leq Cr^{\frac{(2-b)(p+2)-\textbf{p}_c}{2}}(\int|x|^{b}|\nabla(\psi g_n)|^2dx)^{\frac{p+2}{2}}
\end{eqnarray}
for all $2(2-b)<\textbf{p}_c<(2-b)(p+2)$ and $b>2-n$, where we used the Hardy-Sobolev inequality \eqref{c26} at the last step.
This together with \eqref{f69} implies that
\begin{equation}\nonumber
\int_{B_r}\omega(x)|g_n|^{p+2}dx\leq CM^{\frac{p+2}{2}}r^{\frac{(2-b)(p+2)-\textbf{p}_c}{2}}
\end{equation}
for some real number $M>0$.

Next, we will control the integral $\int\omega(x)|g_n|^{p+2}dx$ out of a big ball. For $R>R_1$, we use \eqref{d2} to obtain
\begin{equation}\label{d6}
\int_{\Bbb{R}^n\backslash B_R}\omega(x)|g_n|^{p+2}dx\leq\frac{C}{R^{\frac{2-b}{2}}}\int_{\Bbb{R}^n\backslash B_R}|x|^{c+\frac{2-b}{2}}|g_n|^{p+2}dx.
\end{equation}

To estimate \eqref{d6}, we invoke the Caffarelli-Kohn-Nirenberg inequality \eqref{d5} and take $l=\sigma_1$, $\beta=\frac{\gamma_1}{\sigma_1}$, and $\alpha=\frac{d}{q+2}$ for some fixed real numbers $\gamma_1>-n,\ d>-n$ and $\sigma_1$. Then we obtain
\begin{equation}\label{f67}
\int|x|^d |u|^{q+2}dx\leq C\|\nabla u\|_{b,2}^{a(q+2)}\|u\|_{\gamma_1,\sigma_1}^{(1-a)(q+2)}
\end{equation}
for all $q\in(\frac{(n+d)\sigma_1}{n+\gamma_1}-2,\ \frac{2(2-b+d)}{n-2+b})$ and $b>2-n$,  where
\[a=\frac{\frac{n+d}{q+2}-\frac{n+\gamma_1}{\sigma_1}}{\frac{n-2+b}{2}-\frac{n+\gamma_1}{\sigma_1}}\in(0,1).\]
We note that this inequality \eqref{f67} will be needed in the argument of Proposition \ref{pro1}.
Consequently, if we set $d=c+\frac{2-b}{2}$ and $\sigma_1=\sigma_0$, $\gamma_1=\gamma$ in \eqref{f67}, where $\sigma_0$ and $\gamma$ are defined as \eqref{f64}, then it follows from \eqref{f67} that
\begin{equation}\label{d7}
\{g_n\}_{n\in\Bbb{N}}\ \mbox{is\ uniformly\ bounded\ in}\ L^{q+2}(|x|^{c+\frac{2-b}{2}}dx)
\end{equation}
for all $q\in(\frac{(2n+2c+2-b)\sigma_0}{2(n+\gamma)}-2,\ \frac{2c+3(2-b)}{n-2+b})$, where we used the fact that $\{g_n\}_{n\in\Bbb{N}}$ is uniformly bounded in $\dot{W}_b^{1,2}\cap L^{\sigma_0}(|x|^\gamma dx)$.

Since $2(2-b)<\textbf{p}_c<(2-b)(p+2)$, one can easily check that
\[\frac{(2n+2c+2-b)\sigma_0}{2(n+\gamma)}<p+2<\frac{2n+2c+2-b}{n-2+b},\]
which belongs to the range of $q$ in \eqref{d7}.
Thus, combining \eqref{d6} with \eqref{d7}, we get
\begin{equation}\nonumber
\int_{\Bbb{R}^n\backslash B_R}\omega(x)|g_n|^{p+2}dx\leq \frac{C}{R^{\frac{2-b}{2}}}M^{p+2}.
\end{equation}
Hence, for any $\epsilon>0$, there exist $R>R_1$ and $r_1>r>0$ such that both the estimate \eqref{d10} and
\begin{equation}\label{d11}
\int_{\Bbb{R}^n\backslash B_R}\omega(x)|g_n|^{p+2}dx<\frac{\epsilon}{3}
\end{equation}
hold.

Finally, for $R>r>0$ given above, the remaining task is to prove
\begin{equation}\label{f70}
\int_{B_R\backslash B_r}\omega(x)|g_n|^{p+2}dx<\frac{\epsilon}{3}.
\end{equation}
By Lemma \ref{lem3}, we deduce that
\[\dot{W}_b^{1,2}(B_R\backslash B_r)\hookrightarrow L^{p+2}(B_R\backslash B_r;\ |x|^cdx)\]
is compact for $b>2-n,\ b-2<c\leq \frac{nb}{n-2}$, and $\textbf{p}_c<(2-b)(p+2)$. This together with \eqref{f69} yields that
\begin{equation}\nonumber
g_n\rightarrow0\quad\quad\mbox{in}\ \ L^{p+2}(B_R\backslash B_r;\ |x|^cdx).
\end{equation}
So, \eqref{f70} is proved. Therefore, putting \eqref{d10}, \eqref{d11}, and \eqref{f70} all together, we conclude that the embedding \eqref{d1} is compact. We finish the proof of Lemma \ref{lem4}.
\end{proof}

We remark that the weight function $\omega$ in Lemma \ref{lem4} does not need to be radial.
As a special case, for $\omega(x)=|x|^c$ with $c>\frac{(n-2+b)p}{2}-2+b$, Lemma \ref{lem4} implies that the embedding
\begin{equation}
\dot{W}_b^{1,2}\cap L^{\sigma_0}(|x|^\gamma dx)\hookrightarrow L^{p+2}(|x|^cdx)
\end{equation}
is compact for $\textbf{p}_c<(2-b)(p+2)$.

At the end of this section, we recall the well-known Hardy-Littlewood inequality.
\begin{lemma}\label{lem6}
(Hardy-Littlewood inequality) Let $1<p\leq q<\infty$, $0<s<n$, $\rho\geq0$ satisfy the conditions
\[\rho<\frac{n}{q},\quad\quad s=\frac{n}{p}-\frac{n}{q}+\rho.\]
Then, for any $u\in W^{s,p}$, we have
\[\||x|^{-\rho}u\|_{q}\leq C\|\nabla^su\|_p.\]
\end{lemma}
\begin{proof}
See Theorem $B^\ast$ in \cite{SW2}.
\end{proof}
\begin{rem}
In particular, one has
\[\|u\|_{\gamma,\sigma_0}\leq C\|\nabla^{s}u\|_2, \quad \forall u\in \dot{W}^{s,2}\]
for $b>2-n,\ 0<s_c<\frac{2-b}{2}$ and $-n<\gamma\leq0$, where $\sigma_0$ and $\gamma$ are defined as \eqref{f64}. Note that $\dot{H}^{s_c}\hookrightarrow L^{\sigma_0}(|x|^\gamma dx)$ since $\sigma_0=\frac{2(n+\gamma)}{n-2s_c}$.
\end{rem}

\subsection{A localized virial estimate}
Let $\theta\in C_0^\infty(\Bbb{R}^n)$ be a smooth non-negative radial cutoff function satisfying
\begin{equation}\theta(x)=
\left\{
\begin{aligned}
 & |x|^2,\quad \mbox{if}\ |x|\leq 2,\\\label{f36}
 & 0,\quad\quad \mbox{if}\ |x|\geq4.
\end{aligned}\right. \quad \mbox{and} \quad \theta''(r)\leq2,\quad \mbox{for}\ \ r=|x|.
\end{equation}
By defining $\phi_R(x)=R^2\theta(\frac{x}{R}),$ we can deduce that
\begin{equation}\label{a10}
2-\phi''_R(r)\geq0,\quad 2-\frac{\phi'_R(r)}{r}\geq0, \quad 2n-\Delta\phi_R\geq0.
\end{equation}

Let $u\in C([0, T_\ast); \dot{H}^{s_c}\cap\dot{W}_b^{1,2})$ be a solution to the dINLS \eqref{a0}. For any $R>0$ and $t\in [0, T^\ast]$, we denote
\begin{equation}\label{e5}
V_{\psi_R}(t)=\int\psi_R(x)|u|^2dx,
\end{equation}
and put $\nabla\psi_R=\frac{\nabla\phi_R}{|x|^b}$, then we list the following standard virial identities, which up to standard regularization arguments are obtained by integration by parts on \eqref{a0} (see \cite{ZZ1})
\begin{equation}\label{b26}
V'_{\psi_R}(t)=2\textmd{Im}\int\nabla\phi_R\cdot\nabla u\overline{u}dx,
\end{equation}
and
\begin{eqnarray}\nonumber
\lefteqn{V''_{\psi_R}(t)=-\frac{2}{p+2}\int\left(p|x|^c\Delta\phi_R-2\nabla\phi_R\cdot\nabla|x|^c\right)|u|^{p+2}dx}\\\nonumber
&&\quad\quad\ +4\textmd{Re}\int(\nabla^2\phi_R\cdot\nabla u)\cdot\nabla \overline{u}|x|^bdx-2\int(\nabla\phi_R\cdot\nabla |x|^b)|\nabla u|^2dx\\\label{b11}
&&\quad\quad\ -\int(|x|^b\Delta^2\phi_R+\nabla\Delta\phi_R\cdot\nabla|x|^b)|u|^2dx.
\end{eqnarray}

Action of partial derivatives on radial functions
\[\frac{\partial}{\partial x_j}=\frac{x_j}{r}\partial_r,\quad\quad \frac{\partial^2}{\partial x_j\partial x_k}=\left(\frac{\delta_{jk}}{r}-\frac{x_jx_k}{r^3}\right)\partial_r+\frac{x_jx_k}{r^2}\partial_r^2\]
implies that
\begin{eqnarray}\nonumber
\lefteqn{V''_{\psi_R}(t)=-\frac{2}{p+2}\int\left(p\Delta\phi_R-\frac{2c\partial_r\phi_R}{r}\right)|x|^c|u|^{p+2}dx}\\\nonumber
&&\quad\quad-\int(|x|^b\Delta^2\phi_R+\nabla\Delta\phi_R\cdot\nabla|x|^b)|u|^2dx\\\nonumber
&&\quad\quad+4\int\left(\frac{\partial_r^2\phi_R}{r^2}-\frac{\partial_r\phi_R}{r^3}\right)|x|^b|x\cdot\nabla u|^2dx+(4-2b)\int\frac{\partial_r\phi_R}{r}|x|^b|\nabla u|^2dx.
\end{eqnarray}
Now, using \eqref{f36}, we rewrite the above identity as
\[V''_{\psi_R}(t)=4(2-b)\|\nabla u\|^2_{b,2}-\frac{4\textbf{p}_c}{p+2}\|u\|^{p+2}_{c,p+2}+\mathcal{R}_1+\mathcal{R}_2+\mathcal{R}_3,\quad\quad\quad\]
where
\begin{eqnarray}\nonumber
&&\mathcal{R}_1=-\int_{|x|\geq2R}\left(|x|^b\Delta^2\phi_R+\nabla\Delta\phi_R\cdot\nabla|x|^b\right)|u|^2 dx,\\\nonumber
&&\mathcal{R}_2=4\int_{|x|\geq2R}\left(\frac{\partial_r^2\phi_R}{r^2}-\frac{\partial_r\phi_R}{r^3}\right)|x|^b|x\cdot\nabla u|^2dx\\\nonumber
&&\quad\quad\quad +(4-2b)\int_{|x|\geq2R}\left(\frac{\partial_r\phi_R}{r}-2\right)|x|^b|\nabla u|^2dx,\\\nonumber
&&\mathcal{R}_3=-\frac{2}{p+2}\int_{|x|\geq2R}\left(p\Delta\phi_R-\frac{2c\partial_r\phi_R}{r}\right)|x|^c|u|^{p+2}dx+\frac{4\textbf{p}_c}{p+2}\int_{|x|\geq2 R}|x|^c|u|^{p+2}dx.\quad\quad
\end{eqnarray}
Thus, proceeding similarly to the argument of Lemma 2.8 in \cite{ZZ1}, which deals with the radial solution to \eqref{a0}, we deduce the following localized virial estimate.
\begin{lemma}\label{lem5.2}
Let $2-n<b\leq0$, $c\geq b-2$ and $s_c$ defined as \eqref{f71}. If $u\in C([0, T^\ast);\ \dot{H}^{s_c}\cap\dot{W}_b^{1,2})$ is a solution to \eqref{a0} and $T^\ast$ is the maximal existence time, then there exists a constant $C>0
$ depending only on $b, \textbf{p}_c$ such that for all $t\in [0, T^\ast)$,
\begin{eqnarray}\nonumber
\lefteqn{V_{\psi_R}''(t)\leq 8(ps_c+2-b)E(u_0)-4ps_c\|\nabla u\|^2_{b,2}}\\\label{f4}
&&\quad\quad+C\left(\frac{1}{R^{2-b}}\int_{2R\leq|x|\leq4R}|u|^2dx+\int_{|x|\geq R}|x|^c|u|^{p+2}dx\right).
\end{eqnarray}
\end{lemma}

\section{Blow-up dynamics in $W^{1,2}_b$}
In this section, our aim is to investigate the dynamics of radial blow-up solution to \eqref{a0} in $W^{1,2}_b$, including the upper bound on the blow-up rate and the phenomenon of $L^2$-norm concentration.

We first recall the radial weighted Strauss type inequality in $W^{1,2}_{b}$, which was proved in \cite{ZOZ}.
\begin{lemma}\label{lem3.1}
Let $b\geq2-2n$ and $u\in W_b^{1,2}$ be a radial function. Then one has the following inequality
\begin{equation}\nonumber
\sup_{x\in \Bbb{R}^n}|x|^{\frac{2n-2+b}{4}}|u(x)|\leq C_n\|\nabla u\|_{b,2}^{\frac{1}{2}}\|u\|_{2}^{\frac{1}{2}}.
\end{equation}
\end{lemma}

Based on the localized virial estimate in Lemma \ref{lem5.2}, which is also suitable for the radial case, we give the proof of Theorem \ref{thm1.3}.

\textbf{Proof of Theorem \ref{thm1.3}.}
By Lemma \ref{lem3.1} and the mass conservation \eqref{d26}, we have
\begin{eqnarray}\nonumber
\lefteqn{\int_{|x|\geq R}|x|^c|u|^{p+2}dx\leq R^{\frac{(2-b)p-2\textbf{p}_c}{4}}(\sup_{|x|\geq R}|x|^\frac{2n+b-2}{4}|u|)^p\|u\|_2^2}\\\nonumber
&&\quad\quad\quad\quad\quad\quad\ \leq C_{u_0}R^{\frac{(2-b)p-2\textbf{p}_c}{4}}\|\nabla u\|_{b,2}^{\frac{p}{2}},\quad\quad\quad\quad
\end{eqnarray}
where $C_{u_0}$ depends on the norm $\|u_0\|_2$.
Then, we use the Young inequality to obtain
\begin{equation}\label{f91}
\int_{|x|\geq R}|x|^c|u|^{p+2}dx \leq \epsilon \|\nabla u\|_{b,2}^2+C_{u_0,\epsilon}R^{\frac{(2-b)p-2\textbf{p}_c}{4-p}},
\end{equation}
where we have used the hypothesis of $\textbf{p}_c$ and $p<4.$

Using Lemma \ref{lem5.2}, mass and energy conservations, we deduce
\begin{eqnarray}\nonumber
\lefteqn{2(\textbf{p}_c-4+2b)\int|x|^b|\nabla u|^2dx+V''_{\psi_R}(t)}\\\nonumber
&&\quad\quad\quad\quad\quad\leq C_{u_0}\left(1+\frac{1}{R^{2-b}}+\epsilon\int|x|^b|\nabla u|^2dx+\frac{1}{R^{\frac{2-b}{\beta}}}\right),
\end{eqnarray}
where $\beta=\frac{(4-p)(2-b)}{2\textbf{p}_c-(2-b)p}$.

Let $\epsilon>0$ be small enough such that $\textbf{p}_c-4+2b>\epsilon C_{u_0}$. Therefore if $R\ll1$, since $0<\beta<1$, then there exists a universal constant $C_{u_0}>0$ such that
\begin{equation}\nonumber
(\textbf{p}_c-4+2b)\int|x|^b|\nabla u|^2dx+V''_{\psi_R}(t)\leq\frac{ C_{u_0}}{R^{\frac{2-b}{\beta}}}.
\end{equation}
Now, we integrate the above inequality from $\tau$ to $t$, and then one more time from $\tau$ and $t_1$, since integration by parts yields
\[\int_\tau^{t_1}\int_\tau^t\|\nabla u(s)\|_{b,2}^2dsdt=t_1\int_\tau^{t_1}\|\nabla u(t)\|_{b,2}^2dt-\int_\tau^{t_1}t\|\nabla u(t)\|_{b,2}^2dt,\]
which implies that
\begin{eqnarray}\nonumber
\lefteqn{(\textbf{p}_c-4+2b)\int_\tau^{t_1}(t_1-t)\|\nabla u\|_{b,2}^2dt+V_{\psi_R}(t_1)}\\\label{d24}
&&\ \quad\quad\quad\quad\quad\quad\leq C_{u_0}\frac{(t_1-\tau)^2}{2R^{\frac{2-b}{\beta}}}+V_{\psi_R}(\tau)+(t_1-\tau)V'_{\psi_R}(\tau).
\end{eqnarray}
To control \eqref{d24}, we recall $\nabla\psi_R=\frac{\nabla\phi_R}{|x|^b}$ and $\phi_R(x)=R^2\theta(\frac{x}{R})$, then
\begin{equation}\label{d25}
V_{\psi_R}(\tau)\leq C_b\int_{|x|\leq4R}|x|^{2-b}|u(\tau)|^2dx\leq C_bR^{2-b}\|u\|_2^2.
\end{equation}
Moreover, by the H\"{o}lder inequality and the properties of $\theta$, we have
\begin{eqnarray}\nonumber
\lefteqn{V'_{\psi_R}(\tau)=2\textmd{Im}\int\nabla\phi_R\cdot\nabla u(\tau)\overline{u}(\tau)dx}\\\label{f73}
&&\quad\ \ \leq C(\int R^4|\nabla\theta(\frac{x}{R})|^2|\nabla u|^2dx)^{\frac{1}{2}}\|u\|_2^2\leq CR^{\frac{2-b}{2}}\|\nabla u\|_{b,2}\|u\|_2^2.
\end{eqnarray}
Thus, injecting \eqref{d25} and \eqref{f73} into \eqref{d24}, we get
\begin{eqnarray}\nonumber
\lefteqn{(\textbf{p}_c-4+2b)\int_\tau^{t_1}(t_1-t)\|\nabla u\|_{b,2}^2dt+V_{\psi_R}(t_1)}\\\nonumber
&&\quad\quad\quad\leq C_{b, u_0}\left(\frac{(t_1-\tau)^2}{R^{\frac{2-b}{\beta}}}+R^{2-b}+R^{\frac{2-b}{2}}(t_1-\tau)\|\nabla u\|_{b,2}\right).
\end{eqnarray}
Let $t_1\rightarrow T^\ast$ in the last inequality, we conclude that the first integral on the left hand side converges and
\begin{eqnarray}\label{c20}
\int_\tau^{T^\ast}(T^\ast-t)\|\nabla u\|_{b,2}^2dt\leq C_{b, u_0}\left(\frac{(T^\ast-\tau)^2}{R^{\frac{2-b}{\beta}}}+R^{2-b}+R^{\frac{2-b}{2}}(T^\ast-\tau)\|\nabla
u\|_{b,2}\right).
\end{eqnarray}
Choosing a suitable $R$ so that $R(\tau)=(T^\ast-\tau)^{\frac{2\beta}{(2-b)(\beta+1)}},$ this together with \eqref{c20} yields that
\begin{eqnarray}\nonumber
\lefteqn{\int_\tau^{T^\ast}(T^\ast-t)\|\nabla u\|_{b,2}^2dt\leq C_{b, u_0}\left((T^\ast-\tau)^{\frac{2\beta}{\beta+1}}+(T^\ast-\tau)^{\frac{2\beta+1}{\beta+1}}\|\nabla
u\|_{b,2}\right)}\\\label{c21}
&&\quad\quad\quad\quad\quad\quad\quad\quad\ \leq C_{b, u_0}\left((T^\ast-\tau)^{\frac{2\beta}{\beta+1}}+(T^\ast-\tau)^2\|\nabla u\|^2_{b,2}\right),\quad\quad
\end{eqnarray}
where the last step follows from the Young inequality.

Next, by defining $f(\tau)=\int_\tau^{T^\ast}(T^\ast-t)\|\nabla u\|_{b,2}^2dt$, we observe that
\begin{equation}\nonumber
f(\tau)\leq C_{b, u_0}(T^\ast-\tau)^{\frac{2\beta}{\beta+1}}-(T^\ast-\tau)f'(\tau)
\end{equation}
with $f(0)\leq C_{u_0}$, which also means that
\begin{equation}\nonumber
\left(\frac{f(\tau)}{T^\ast-\tau}\right)'\leq C_{b,u_0}(T^\ast-\tau)^{-\frac{2}{\beta+1}}.
\end{equation}
Integrating this differential inequality from 0 to $\tau$, we have
\begin{eqnarray}\nonumber
\quad\quad\quad\frac{f(\tau)}{T^\ast-\tau}\leq \frac{C_{b, u_0}}{(T^\ast-\tau)^{\frac{1-\beta}{1+\beta}}}-\frac{C_{b, u_0}}{(T^\ast)^{\frac{1-\beta}{1+\beta}}}+\frac{f(0)}{T^\ast}\leq\frac{C_{b, u_0}}{(T^\ast-\tau)^{\frac{1-\beta}{1+\beta}}}
\end{eqnarray}
for $\tau$ close enough to $T^\ast$. Hence, \eqref{c23} is proved, and we finish the proof of Theorem \ref{thm1.3}.\quad\quad\quad\quad\quad\quad\quad\quad\quad\quad\quad\quad\quad\quad\quad\quad\quad\quad\quad\quad\quad\quad\quad\quad\quad\quad\quad\quad\quad\quad\quad\quad\quad\quad $\square$

Next, we focus on the mass-critical case $s_c=0$, and study the $L^2$-norm concentration of radial blow-up solution to \eqref{a0}. The main ingredient of the proof of Theorem \ref{thm1.1} is the following compact weighted Sobolev embedding lemma for radial function in $W_b^{1,2}$, which was established in \cite{SW}.
\begin{lemma}\label{lem1}
Let $b<2$ and $d\geq b-2$. Then the following statements hold.

(1) If $d<0$, then the weighted Sobolev embedding for radial functions
\begin{equation}\nonumber
W_b^{1,2}\hookrightarrow L^q(|x|^ddx)
\end{equation}
is compact for $2<q<\frac{2d+2n}{n-2+b}$.

(2) If $d\geq 0$, then the weighted Sobolev embedding for radial functions
\begin{equation}\nonumber
W_b^{1,2}\hookrightarrow L^q(|x|^ddx)
\end{equation}
is also compact for $\frac{4d}{2n-2+b}+2< q<\frac{2d+2n}{n-2+b}$.
\end{lemma}

Together with the radial Gagliardo-Nirenberg inequality \eqref{d21}, it will allow us to derive Theorem \ref{thm1.1} as follows.

\textbf{Proof of Theorem \ref{thm1.1}.} Let $\{t_n\}_{n\in\Bbb{N}}$ be an arbitrary time sequence such that $t_n\rightarrow T^\ast$ as $n\rightarrow\infty$. We define
\[\rho_n:=\rho(t_n)=\frac{\|\nabla Q\|_{b,2}}{\|\nabla u(t_n)\|_{b,2}}\]
and
\[v_n(x):=v(x,t_n)=\rho_n^{\frac{n}{2-b}}u(\rho_n^{\frac{2}{2-b}}x, t).\]
Thus the scaling lower bound \eqref{c8} implies $\rho_n\rightarrow0$ as $n\rightarrow\infty$. In this case, for all $n\in\Bbb{N}$, we easily get
\begin{equation}\nonumber
\|v_n\|_2=\|u_0\|_2, \quad\quad\quad \|\nabla v_n\|_{b,2}=\|\nabla Q\|_{b,2}.
\end{equation}
Since $p=p_{b,\ast}$, one can see that
\begin{equation}\nonumber
E(v_n)=\frac{1}{2}\|\nabla v_n\|^2_{b,2}-\frac{1}{p_{b,\ast}+2}\|v_n\|^{p_{b,\ast}+2}_{c,p_{b,\ast}+2}=\rho_n^2E(u(t_n)),
\end{equation}
so $\{v_n\}_{n\in\Bbb{N}}$ is a bounded sequence in $W_b^{1,2}$. Moreover, $\rho_n\rightarrow0$ as $n\rightarrow\infty$, and the energy conservation \eqref{c10} yield
\[\lim_{n\rightarrow\infty}E(v_n)=0.\]
Thus, by Lemma \ref{lem1}, there exists a radial function $v^\ast\in L^q(|x|^ddx)$, up to a subsequence, such that
\begin{equation}\nonumber
v_n\rightharpoonup v^\ast \quad\quad \mbox{in}\ \ W_{b}^{1,2},
\end{equation}
and
\begin{equation}\label{c3}
\quad\quad v_n\rightarrow v^\ast \quad\quad \mbox{in}\ \  L^q(|x|^ddx)
\end{equation}
are valid for one of the following conditions:
\begin{eqnarray}\nonumber
&&\ (i)\ \  b-2\leq d<0:\quad\quad 2<q<\frac{2d+2n}{n-2+b},\\\nonumber
&&\ (ii)\ \ d\geq0:\quad\quad\quad\ \frac{4d}{2n-2+b}+2<q<\frac{2d+2n}{n-2+b}.\quad\quad\quad\quad\quad\quad\quad\quad\quad\quad
\end{eqnarray}
Picking $d=c$ in \eqref{c3}, by hypothesis of $c$, one can easily verify that the real number $p_{b,\ast}+2$ is within the range of $q$ in (i) or (ii). So we have
\begin{equation}\label{e1}
\lim_{n\rightarrow\infty}\int |x|^{c}|v_n|^{p_{b,\ast}+2}dx=\int |x|^{c}|v^\ast|^{p_{b,\ast}+2}dx.
\end{equation}
Hence, applying the radial Gagliardo-Nirenberg inequality \eqref{d21} with $p=p_{b,\ast}$, we deduce from \eqref{e1} that
\begin{eqnarray}\nonumber
\lefteqn{0=\lim_{n\rightarrow\infty}E(v_n)\geq \frac{1}{2}\|\nabla
v^\ast\|^2_{b,2}-\frac{C_{GN}}{p_{b,\ast}+2}\|\nabla v^\ast\|^{2}_{b,2}\|v^\ast\|^{p}_2}\\\nonumber
&&\quad\quad\quad\quad\quad=\frac{1}{2}\|\nabla v^\ast\|^2_{b,2}\left(1-\frac{\|v^\ast\|^{p}_2}{\|Q\|^{p}_2}\right),\quad\quad\quad\quad\quad
\end{eqnarray}
which leads to $\|v^\ast\|_2\geq \|Q\|_2$. Now, for any $A>0$, we obtain
\begin{eqnarray}\nonumber
\lefteqn{\liminf_{n\rightarrow\infty}\int_{|y|\leq A\rho_n^{\frac{2}{2-b}}} |u(y, t_n)|^2dy=\liminf_{n\rightarrow\infty}\int_{|x|\leq A} \rho_n^{\frac{2n}{2-b}}|u(\rho_n^{\frac{2}{2-b}}x, t_n)|^2dx}\\\nonumber
&&\quad\quad\quad\quad\quad\quad\quad\quad\quad\quad\quad\ =\liminf_{n\rightarrow\infty}\int_{|x|\leq A} |v_n(x)|^2dx\geq\int_{|x|\leq A} |v^\ast|^2dx.\quad\quad
\end{eqnarray}
Noting that
\[\frac{\lambda(t_n)}{\rho_n^{\frac{2}{2-b}}}=\frac{\lambda(t_n)\|\nabla u(t_n)\|_{b,2}^{\frac{2}{2-b}}}{\|\nabla Q\|_{b,2}^{\frac{2}{2-b}}}\rightarrow\infty\quad \mbox{as}\ \
n\rightarrow\infty,\]
we then get
\[\liminf_{n\rightarrow\infty}\int_{|x|\leq\lambda(t_n)} |u(x, t_n)|^2dx\geq\int_{|x|\leq A} |Q|^2dx,\]
for any $A>0$. Since $\{t_n\}$ is arbitrary, we deduce that
\[\liminf_{t\rightarrow T^\ast}\int_{|x|\leq\lambda(t)} |u(x, t)|^2dx\geq\|Q\|^2_2.\]
This arrives at \eqref{c7}.

Next, we turn to the proof of part (2) by contradiction. Assume that there exists a sequence $\{t_n\}\subset[0, T^\ast)$ with $t_n\rightarrow T^\ast$ as $n\rightarrow\infty$, such that $u(\cdot,t_n)$ converges strongly in $L^2$. By the energy conservation \eqref{c10}, for any $n, m>1$, we get
\begin{eqnarray}\nonumber
\lefteqn{\frac{1}{2}\|\nabla u(t_n)\|_{b,2}^{2}\leq\frac{1}{p_{b,\ast}+2}\|u(t_n)\|_{c,p_{b,\ast}+2}^{p_{b,\ast}+2}+|E(u_0)|}\\\label{c27}
&&\quad\quad\quad\quad\leq C\|u(t_n)-u(t_m)\|_{c,p_{b,\ast}+2}^{p_{b,\ast}+2}+C\|u(t_m)\|_{c,p_{b,\ast}+2}^{p_{b,\ast}+2}+|E(u_0)|,\quad\quad\
\end{eqnarray}
where the constant $C$ depends on $\textbf{p}_c$. Observe that the Gagliardo-Nirenberg inequality \eqref{d21} implies
\begin{eqnarray}\nonumber
\lefteqn{\|u(t_n)-u(t_m)\|_{c,p_{b,\ast}+2}^{p_{b,\ast}+2}\leq C\|\nabla u(t_n)-\nabla u(t_m)\|_{b,2}^{2}\|u(t_n)-u(t_m)\|_{2}^{p_{b,\ast}}}\\\label{c28}
&&\quad\quad\quad\quad\quad\quad\quad\quad\ \leq C\|u(t_n)-u(t_m)\|_{2}^{p_{b,\ast}}\|\nabla u(t_n)\|_{b,2}^2+C_m.\quad
\end{eqnarray}
Thus, plugging \eqref{c28} into \eqref{c27}, we have
\begin{equation}\nonumber
\frac{1}{2}\|\nabla u(t_n)\|_{b,2}^{2}\leq C\|u(t_n)-u(t_m)\|_{2}^{p_{b,\ast}}\|\nabla u(t_n)\|_{b,2}^2+C_m.
\end{equation}
Since $u(\cdot,t_n)$ converges strongly in $L^2$, for any $n$ large enough and some large number $m_0$, we have $C\|u(t_n)-u(t_m)\|_{2}^{p_{b,\ast}}<\frac{1}{4}$. Hence we obtain $\|\nabla u(t_n)\|_{b,2}\leq C_{m_0}$, where $C_{m_0}$ is independent of $n$, which contradicts the fact that $u(t)$ blows up at $T^\ast$. The proof of Theorem \ref{thm1.1} is completed.

\section{Two types of Gagliardo-Nirenberg inequalities}
In this section, we will establish two type of Gagliardo-Nirenberg inequalities for general functions in various spaces.

To begin with, we recall the standard Gagliardo-Nirenberg inequality
\begin{equation}\label{f63}
\int|x|^c|f|^{p+2}dx\leq C\|\nabla f\|_{b,2}^{\frac{\textbf{p}_c}{2-b}}\|f\|_2^{p+2-\frac{\textbf{p}_c}{2-b}},\quad\quad \forall\ f\in W^{1,2}_b,
\end{equation}
which occurred in Remark \ref{rem1}.

For our purpose, we first prove an improved version of \eqref{f63} for general functions in $\dot{W}_b^{1,2}\cap L^{\sigma_0}(|x|^\gamma dx)$
\begin{equation}\nonumber
\int|x|^c|f|^{p+2}dx\leq C_{GN}\|\nabla f\|_{b,2}^2\|f\|_{\gamma,\sigma_0}^p,
\end{equation}
which will be given in Section \ref{sec1}, and extends the one obtained by \cite{CC} for the INLS equation to the dINLS case with $b>2-n$.

In Section \ref{sec2}, we will develop a refined Gagliardo-Nirenberg inequality for general functions in $\dot{H}^{s_c}\cap\dot{W}_b^{1,2}$, which is based on the following scaling invariant Morrey-Campanato type functional $\rho$ used in \cite{MR}, defined by
\begin{equation}\label{f26}
\rho(u, R)=\sup_{R'\geq R}\frac{1}{(R')^{2s_c}}\int_{R'\leq|x|\leq 2R'}|u|^2dx,
\end{equation}
replacing the role of the $L^2$ norm in \eqref{f63}. We point out that no radial symmetry is needed to obtain these inequalities, which are essential to the proof of Theorem \ref{thm2}.

\subsection{Gagliardo-Nirenberg inequality in $\dot{W}_b^{1,2}\cap L^{\sigma_0}(|x|^\gamma dx)$}\label{sec1}
In the spirit of Weinstein's \cite{MW2}, we establish the following sharp Gagliardo-Nirenberg inequality in $\dot{W}_b^{1,2}\cap L^{\sigma_0}(|x|^\gamma dx)$ by the variational method and obtain the best constant for this inequality.
\begin{lemma}\label{pro1}
Let  $2-n<b<2,\ b-2<c\leq\frac{nb}{n-2},\ \gamma>-n$ and $2(2-b)<\textbf{p}_c<(2-b)(p+2)$. Then for any $f\in \dot{W}_b^{1,2}\cap L^{\sigma_0}(|x|^\gamma dx)$ with $\sigma_0,\ \gamma$ defined as \eqref{f64}, we have
\begin{equation}\label{d15}
\int|x|^c|f|^{p+2}dx\leq C_{GN}\|\nabla f\|_{b,2}^2\|f\|_{\gamma,\sigma_0}^p,
\end{equation}
where the best constant $C_{GN}=\frac{p+2}{2\|\mathcal{Q}\|_{\gamma,\sigma_0}^p}$, and $\mathcal{Q}$ is the solution to the following elliptic equation
\begin{equation}\label{d16}
\nabla\cdot(|x|^b\nabla \mathcal{Q})+|x|^c|\mathcal{Q}|^{p}\mathcal{Q}-|x|^\gamma|\mathcal{Q}|^{\sigma_0-2}\mathcal{Q}=0
\end{equation}
with minimal $L^{\sigma_0}(|x|^\gamma dx)$-norm.
\end{lemma}
\begin{proof}
We define the variational problem
\begin{equation}\label{a20}
J_\ast=\inf\{J(f):\ f\in \dot{W}_b^{1,2}\cap L^{\sigma_0}(|x|^\gamma dx)\},
\end{equation}
where the Weinstein functional $J(f)$ is defined by
\[J(f)=\frac{\|\nabla f\|_{b,2}^2\|f\|_{\gamma,\sigma_0}^p}{\|f\|_{c,p+2}^{p+2}}.\]

First, we claim that the minimizer of variational problem $J(f)$ can be attained.

Indeed, picking $q=p,\ d=c$, and $\sigma_1=\sigma_0$, $\gamma_1=\gamma$ in the inequality \eqref{f67}, we obtain
\begin{equation}\label{d13}
\int|x|^c |f|^{p+2}dx\leq C\|\nabla f\|_{b,2}^2\|f\|_{\gamma,\sigma_0}^p
\end{equation}
for all $p<\frac{2(2-b+c)}{n-2+b}$ and $b>2-n$, $\gamma>-n$, $c>2-b$. So, the variational problem \eqref{a20} is well-defined.

Now, choosing a minimizing sequence $\{f_n\}_{n\in\Bbb{N}}\subset \dot{W}_b^{1,2}\cap L^{\sigma_0}(|x|^\gamma dx)$ in \eqref{a20} such that
\begin{equation}\nonumber
\lim_{n\rightarrow\infty}J(f_n)=J_\ast.
\end{equation}
We rescale $\{f_n\}_{n\in\Bbb{N}}$ by setting
\[g_n(x)=\mu_n f_n(\lambda_n x),\] with
\[\mu_n=\frac{\|f_n\|_{\gamma,\sigma_0}^{\frac{n-2+b}{2-b-2s_c}}}{\|\nabla f_n\|_{b,2}^{\frac{2(2-b+c)}{p(2-b-2s_c)}}},\quad\quad \lambda_n=\left(\frac{\|f_n\|_{\gamma,\sigma_0}}{\|\nabla f_n\|_{b,2}}\right)^{\frac{2}{2-b-2s_c}},\]
so that
\begin{equation}\label{a25}
\|g_n\|_{\gamma,\sigma_0}=1\quad\quad\mbox{and}\quad\quad \|\nabla g_n\|_{b,2}=1.
\end{equation}
Moreover, since $J$ is invariant under this scaling, $\{g_n\}_{n\in\Bbb{N}}$ is also a minimizing sequence of \eqref{a20}, which is bounded in $\dot{W}_b^{1,2}\cap L^{\sigma_0}(|x|^\gamma dx)$.
Thus, up to a subsequence, there exists a function $\widetilde{g}\in \dot{W}_b^{1,2}\cap L^{\sigma_0}(|x|^\gamma dx)$ such that
\begin{equation}\nonumber
g_n\rightharpoonup \widetilde{g}\quad \mbox{weakly\ in\ } \dot{W}_b^{1,2}\cap L^{\sigma_0}(|x|^\gamma dx).
\end{equation}
It follows from \eqref{a25} that
\begin{equation}\label{a26}
\|\widetilde{g}\|_{\gamma,\sigma_0}\leq1 \quad\quad\mbox{and}\quad\quad \|\nabla \widetilde{g}\|_{b,2}\leq1.
\end{equation}
Meanwhile, invoking the compactness result in Lemma \ref{lem4}, we also deduce that
\begin{equation}\label{a27}
g_n\rightarrow \widetilde{g}\quad \mbox{strongly\ in\ } L^{p+2}(|x|^cdx).
\end{equation}
Therefore, collecting \eqref{a26} and \eqref{a27}, we obtain
\begin{equation}\nonumber
J_\ast\leq J(\widetilde{g})\leq\frac{1}{\|\widetilde{g}\|_{c,p+2}^{p+2}}=\lim_{n\rightarrow\infty}\frac{1}{\|g_n\|_{c,p+2}^{p+2}}=J_\ast.
\end{equation}
Consequently, we get
\begin{equation}\nonumber
\|\widetilde{g}\|_{\gamma,\sigma_0}=1,\quad \|\nabla \widetilde{g}\|_{b,2}=1\quad \mbox{and}\quad J_\ast=\frac{1}{\|\widetilde{g}\|_{c,p+2}^{p+2}}.
\end{equation}
Hence, $\widetilde{g}$ is a minimizer for the functional $J$.

Next, we claim that the best constant of \eqref{d13} is
\begin{equation}\label{a2}
C_{GN}=\frac{p+2}{2\|\mathcal{Q}\|_{\gamma,\sigma_0}^p}.
\end{equation}

In fact, since $\widetilde{g}$ is the minimizer of the functional $J$, then $\widetilde{g}$ is a solution to the following Euler-Lagrange equation
\begin{equation}\label{d14}
\frac{d}{d\epsilon}J(\widetilde{g}+\epsilon \varphi)|_{\epsilon=0}=0\quad\quad \mbox{for\  all\ }\varphi\in C_0^\infty(\Bbb{R}^n).
\end{equation}
A simple computation shows that
\[\frac{d}{d\epsilon}\|\nabla(\widetilde{g}+\epsilon\varphi)\|_{b,2}^2|_{\epsilon=0}=2\textmd{Re}\langle-\nabla\cdot(|x|^b\nabla\widetilde{g}),\ \varphi\rangle,\]
\[\frac{d}{d\epsilon}\|\widetilde{g}+\epsilon\varphi\|_{c,p+2}^{p+2}|_{\epsilon=0}=(p+2)\textmd{Re}\langle|x|^c|\widetilde{g}|^p\widetilde{g},\ \varphi\rangle,\]
and
\begin{eqnarray}\nonumber
\lefteqn{\frac{d}{d\epsilon}\|\widetilde{g}+\epsilon\varphi\|_{\gamma,\sigma_0}^p|_{\epsilon=0}=\frac{p}{\sigma_0}\|\widetilde{g}\|_{\gamma,\sigma_0}^{p-\sigma_0}\frac{d}{d\epsilon}\|\widetilde{g}+\epsilon\varphi\|_{\gamma,\sigma_0}^{\sigma_0}|_{\epsilon=0}}\\\nonumber
&&\quad\quad\quad\quad\quad\quad\ \
 =p\|\widetilde{g}\|_{\gamma,\sigma_0}^{p-\sigma_0}\textmd{Re}\langle|x|^\gamma|\widetilde{g}|^{\sigma_0-2}\widetilde{g},\ \varphi\rangle,
\end{eqnarray}
where $\langle\cdot,\ \cdot\rangle$ is the inner product in $L^2$. Then it follows from \eqref{d14} that
\begin{eqnarray}\nonumber
\lefteqn{-2\textrm{Re}\int\nabla\cdot(|x|^b\nabla \widetilde{g})\overline{\varphi}dx-(p+2)J_\ast\textrm{Re}\int |x|^c|\widetilde{g}|^p\widetilde{g}\overline{\varphi}dx}\\\nonumber
&&\quad\quad\quad+p\textrm{Re}\int|x|^\gamma|\widetilde{g}|^{\sigma_0-2}\widetilde{g}\overline{\varphi}dx=0,\quad\quad\quad\quad\quad\quad\quad
\end{eqnarray}
which implies that
\begin{equation}\nonumber
\nabla\cdot(|x|^b\nabla \widetilde{g})-\frac{p}{2}|x|^\gamma|\widetilde{g}|^{\sigma_0-2}\widetilde{g}+\frac{p+2}{2}J_\ast|x|^c|\widetilde{g}|^p\widetilde{g}=0.
\end{equation}

Now we choose $\mathcal{Q}$ and define $\widetilde{g}(x)=\alpha \mathcal{Q}(\beta x),$ with
\[\alpha=\left(\frac{p}{2}\right)^{-\frac{n+\gamma}{\Lambda}}\left(\frac{(p+2)J_\ast}{2}\right)^{\frac{(\gamma+2-b)(n+\gamma)}{(2-b+c)\Lambda}}\]
and \[\beta=\left(\frac{p}{2}\right)^{\frac{\Lambda-(\sigma_0-2)(n+\gamma)}{(\gamma+2-b)\Lambda}}\left(\frac{(p+2)J_\ast}{2}\right)^{\frac{(\sigma_0-2)(n+\gamma)}{(2-b+c)\Lambda}},\]
where $\Lambda:=(n+\gamma)(\sigma_0-2)-(\gamma+2-b)\sigma_0$. It can be calculated that $\mathcal{Q}$ solves the following equation
\begin{equation}\nonumber
\nabla\cdot(|x|^b\nabla \mathcal{Q})-|x|^\gamma|\mathcal{Q}|^{\sigma_0-2}\mathcal{Q}+|x|^c|\mathcal{Q}|^p\mathcal{Q}=0
\end{equation}
and
\[\|\mathcal{Q}\|_{\gamma, \sigma_0}^{\sigma_0}=\left(\frac{p+2}{2}J_\ast\right)^{\frac{n+\gamma}{2-b+c}}.\]
This implies that
\begin{equation}\nonumber
J_\ast=\frac{2\|\mathcal{Q}\|_{\gamma,\sigma_0}^p}{p+2},
\end{equation}
which arrives at \eqref{a2}. Hence, by the definition of $J_\ast$, we have for all $ f\in \dot{W}_b^{1,2}\cap L^{\sigma_0}(|x|^\gamma dx)$,
\[\frac{2\|\mathcal{Q}\|_{\gamma,\sigma_0}^p}{p+2}=J_\ast\leq\frac{\|\nabla f\|_{b,2}^2\|f\|_{\gamma,\sigma_0}^p}{\int|x|^c|f|^{p+2}dx}.\]
The proof of Proposition \ref{pro1} is completed.
\end{proof}
\begin{rem}
(1) In the limiting case $\textbf{p}_c=2(2-b)$, the best constant in Lemma \ref{pro1} was already established in \cite{ZOZ} for radial symmetric functions in $W_b^{1,2}$, in the case $b\leq0$.

(2) Although the uniqueness of solutions for the elliptic equation \eqref{d16} is still unclear, it will not be an issue to our purpose, since the sharp constant in Lemma  \ref{pro1} depends only on the $L^{\sigma_0}(|x|^\gamma dx)$-norm of the solution $\mathcal{Q}$.
\end{rem}

\subsection{Gagliardo-Nirenberg inequality in $\dot{H}^{s_c}\cap\dot{W}_b^{1,2}$}\label{sec2}
In this part, we will prove a crucial Gagliardo-Nirenberg inequality for general non-radial functions in $\dot{H}^{s_c}\cap\dot{W}_b^{1,2}$.
\begin{lemma}\label{lem5.1}
Let $n\geq3,\ 2-n<b<2$, $b-2<c<\frac{nbp}{4}$, $2(2-b)<\textbf{p}_c<(2-b)(p+2)$ and $p<\frac{4}{n}$. If $\sigma_0$ is defined as \eqref{f64} with $\gamma>-n$, Then

(i) There exists a universal constant $C>0$ such that for all $u\in L^{\sigma_0}(|x|^\gamma dx)$,
\begin{equation}\label{f1}
\frac{1}{R^{2s_c}}\int_{|x|\leq R}|u|^2dx\leq C\|u\|^2_{\gamma, \sigma_0},\quad \forall\ R>0
\end{equation}
and
\begin{equation}\label{f3}
\lim_{R\rightarrow+\infty}\frac{1}{R^{2s_c}}\int_{|x|\leq R}|u|^2dx=0.
\end{equation}

(ii) For any $\eta>0$, there exists a constant $C_\eta>0$ such that for all $u\in \dot{H}^{s_c}\cap\dot{W}_b^{1,2}$ and $R>0$, the following inequality
\begin{equation}\label{f5}
\int_{|x|\geq R}|x|^c|u|^{p+2}dx\leq \eta\|\nabla u\|_{b,2}^2+\frac{C_\eta}{R^{2-b-2s_c}}\rho(u, R)^{\frac{4-np+2p}{4-np}}
\end{equation}
holds, where $\rho(u, R)$ is defined as \eqref{f26}.
\end{lemma}

\begin{proof}
(i) By the H\"{o}lder inequality, we have
\begin{equation}\nonumber
\int_{|x|\leq R}|u|^2dx\leq (\int_{|x|\leq R}|x|^\gamma|u|^{\sigma_0}dx)^{\frac{2}{\sigma_0}}(\int_{|x|\leq R}|x|^{\frac{2\gamma}{2-\sigma_0}}dx)^{1-\frac{2}{\sigma_0}},
\end{equation}
which implies \eqref{f1}.

Let $R>A>1$, we split the above integral into two parts and use \eqref{f1} to obtain
\begin{eqnarray}\nonumber
\lefteqn{\frac{1}{R^{2s_c}}\int_{|x|\leq R}|u|^2dx\leq \frac{1}{R^{2s_c}}\int_{|x|\leq A}|u|^2dx+\int_{A\leq|x|\leq R}|u|^2dx}\\\label{f2}
&&\quad\quad\quad\quad\quad\quad\leq C\left(\frac{A}{R}\right)^{2s_c}\|u\|^2_{\gamma, \sigma_0}+\int_{A\leq|x|\leq R}|u|^2dx.
\end{eqnarray}
Given $\varepsilon>0$, choosing $R>A>1$ large enough such that
\[\int_{|x|\geq A}|u|^2dx<\frac{\varepsilon}{2},\]
which together with \eqref{f2} yields \eqref{f3}.

(ii) For each $j\in\Bbb{N}$, we set $\mathcal{C}_j=\{x\in\Bbb{R}:\ 2^jR\leq|x|\leq 2^{j+1}R\}$, then
\begin{equation}\label{f12}
\int_{|x|\geq R}|x|^c|u|^{p+2}dx=\sum_{j=0}^\infty\int_{\mathcal{C}_j}|x|^c|u|^{p+2}dx.
\end{equation}

Now we recall the following weighted interpolation inequality
\begin{equation}\nonumber
\left(L^{p_0}(w_0^{p_0}dx),\ L^{p_1}(w_1^{p_1}dx)\right)_{\vartheta, \widetilde{q}}=L^{\widetilde{q}}(w_0^{(1-\vartheta)\widetilde{q}}w_1^{\vartheta\widetilde{q}}dx),
\end{equation}
cited in \cite{DF}, where $\frac{1}{\widetilde{q}}=\frac{1-\vartheta}{p_0}+\frac{\vartheta}{p_1}$ and $0<\vartheta<1$, which can be applied to get
\begin{equation}\label{f9}
\int_{\mathcal{C}_j}|x|^{\frac{nbp}{4}}|u|^{p+2}dx\leq C(\int_{\mathcal{C}_j}|x|^{\frac{nb}{n-2}}|u|^{\frac{2n}{n-2}}dx)^{\frac{(n-2)p}{4}}(\int_{\mathcal{C}_j}|u|^2dx)^{\frac{4-np+2p}{4}}.
\end{equation}
Thus, by hypothesis and \eqref{f9}, we deduce that
\begin{eqnarray}\nonumber
\lefteqn{\int_{\mathcal{C}_j}|x|^c|u|^{p+2}dx\leq\frac{1}{(2^jR)^{\frac{nbp}{4}-c}}\int_{\mathcal{C}_j}|x|^{\frac{nbp}{4}}|u|^{p+2}dx}\\\nonumber
&&\quad\quad\quad\quad\quad\leq \frac{C}{(2^jR)^{\frac{nbp}{4}-c}}(\int_{\mathcal{C}_j}|x|^{\frac{nb}{n-2}}|u|^{\frac{2n}{n-2}}dx)^{\frac{(n-2)p}{4}}(\int_{\mathcal{C}_j}|u|^2dx)^{\frac{4-np+2p}{4}}.
\end{eqnarray}
Applying the Hardy-Sobolev inequality \eqref{c26} to the last inequality, we find that
\begin{eqnarray}\nonumber
\lefteqn{\int_{\mathcal{C}_j}|x|^c|u|^{p+2}dx\leq\frac{C_b}{(2^jR)^{\frac{nbp}{4}-c}}\|\nabla u\|_{b,2}^{\frac{np}{2}}(\int_{\mathcal{C}_j}|u|^2dx)^{\frac{4-np+2p}{4}}}\\\nonumber
&&\quad\quad\quad\quad\quad\leq C_b\|\nabla u\|_{b,2}^{\frac{np}{2}}\frac{1}{(2^jR)^{\frac{npb}{4}-c-2s_c(1+\frac{2p-np}{4})}}\rho(u, 2^jR)^{\frac{4-np+2p}{4}},\quad
\end{eqnarray}
where
\[\frac{npb}{4}-c-2s_c(1+\frac{2p-np}{4})=\frac{4-np}{4}(2-b-2s_c)>0,\]
which holds for all $s_c<\frac{2-b}{2}$ and $p<\frac{4}{n}$.

Now let $\iota\in (0,\ \frac{4-np}{4}(2-b-2s_c))$, by the Young inequality, there exists an $\widetilde{\eta}>0$ such that
\begin{equation}\label{f13}
\int_{\mathcal{C}_j}|x|^c|u|^{p+2}dx\leq\frac{\widetilde{\eta}}{2^{\frac{4\iota j}{np}}}\|\nabla u\|_{b,2}^{2}+\frac{C_{\widetilde{\eta}}}{(2^jR)^{2-b-2s_c}2^{\frac{-4\iota j}{4-np}}}\rho(u, 2^jR)^{\frac{4-np+2p}{4-np}}.
\end{equation}
Since $\rho(u, R)$ is non-increasing in $R$, we have $\rho(u, 2^jR)\leq \rho(u, R)$ for any $j\in\Bbb{R}$.
So, combining \eqref{f12} with \eqref{f13}, we obtain
\begin{eqnarray}\nonumber
\lefteqn{\int_{|x|\geq R}|x|^c|u|^{p+2}dx\leq \widetilde{\eta}\|\nabla u\|_{b,2}^{2}\sum_{j=0}^\infty\frac{1}{(2^{\frac{4\iota}{np}})^j}}\\\nonumber
&&\quad\quad\quad\quad\quad\quad\quad+\frac{C_{\widetilde{\eta}}}{R^{2-b-2s_c}}\rho(u, R)^{\frac{4-np+2p}{4-np}}\sum_{j=0}^\infty\frac{1}{(2^{2-b-2s_c-\frac{4\iota}{4-np}})^j},
\end{eqnarray}
where the above series are summable and we conclude the estimate \eqref{f5}.
\end{proof}

\begin{rem}
The radial Gagliardo-Nirenberg inequality for functions in $\dot{H}^{s_c}\cap\dot{W}_b^{1,2}$ can also be established as follows:
\begin{equation}\nonumber
\int_{|x|\geq R}|x|^c|u|^{p+2}dx\leq \eta\int_{|x|\geq R}|x|^b|\nabla u|^2dx+\frac{C_\eta}{R^{2-b-2s_c}}\left(\rho(u, R)^{\frac{4+p}{4-p}}+\rho(u, R)^{\frac{p+2}{2}}\right).
\end{equation}
\end{rem}
The argument is quite similar to the one proved in Proposition 4.2 in \cite{CF}, so we omit the details.

\section{Main propositions}
In this section, we are devoted to the proof of some uniform estimates that are the main ingredients in the proof of Theorem \ref{thm2}. The technique is very similar to the one used in \cite{MR}, where only the radial solution is treated for the classical NLS equation. In our work, we show that Lemmas \ref{lem5.2}, \ref{pro1} and \ref{lem5.1} allow us to consider more general non-radial solutions to the dINLS equation.

First, we deduce a priori control of the functional $\rho$ defined in \eqref{f26} on parabolic space time intervals for a suitable choice of $R>0$.
\begin{lemma}\label{lem7}
Let $n\geq3,\ 2-n<b\leq0,\ b-2<c<\frac{nbp}{4},\ 2(2-b)<\textbf{p}_c<(2-b)(p+2),\ p<\frac{4}{n}$ and $\sigma_0,\ \gamma$ be defined as \eqref{f64} with $\gamma>-n$. Let $u\in C([0, \tau_\ast];\ \dot{H}^{s_c}\cap\dot{W}_b^{1,2})$ be a solution to \eqref{a0} with initial data $u_0\in\dot{H}^{s_c}\cap\dot{W}_b^{1,2}$. For any $A>0$ and $\tau_0\in[0, \tau_\ast]$, let $R=A\tau_0^{\frac{1}{2-b}}$ and define $M_\infty$ by
\begin{equation}\label{f27}
M_\infty^2(A, \tau_0)=\max_{\tau\in[0, \tau_0]}\rho(u(\tau), A\tau^{\frac{1}{2-b}}).
\end{equation}
Then, there exists a universal constant $C>0$ such that
\begin{eqnarray}\nonumber
\lefteqn{4ps_c\int_0^{\tau_0}(\tau_0-\tau)\|\nabla u(\tau)\|_{b,2}^2d\tau}\\\nonumber
&&\quad\leq \tau_0\left(V'_{\psi_R}(0)+4(ps_c+2-b)E(u_0)\tau_0\right)\\\label{f17}
&&\quad+C\tau_0^{1+\frac{2s_c}{2-b}}\left(A^{2-b+2s_c}\|u_0\|^2_{\gamma, \sigma_0}+\frac{M_\infty(A, \tau_0)^{\frac{2(4-np+2p)}{4-np}}+M_\infty^2(A, \tau_0)}{A^{2-b-2s_c}}\right)\quad\quad
\end{eqnarray}
and
\begin{eqnarray}\nonumber
\lefteqn{\frac{1}{R^{2s_c}}\int_{R\leq|x|\leq2R}|u(\tau_0)|^2dx}\\\nonumber
&&\leq \frac{1}{\tau_0^{\frac{2s_c}{2-b}}A^{2-b+2s_c}}\left(V'_{\psi_R}(0)+4(ps_c+2-b)E(u_0)\tau_0\right)\\\label{f18}
&&\quad+C\|u_0\|^2_{\gamma, \sigma_0}+\frac{C}{A^{2(2-b)}}\left(M_\infty(A, \tau_0)^{\frac{2(4-np+2p)}{4-np}}+M^2_\infty(A, \tau_0)\right).\quad\quad
\end{eqnarray}
\end{lemma}
\begin{proof}
We consider the virial estimate in Lemma \ref{lem5.2} with $R=A\tau_0^{\frac{1}{2-b}}$ and estimate the terms on the right hand side.
Since $\rho(u, R)$ is non-increasing in $R$, we have
\begin{equation}\label{f6}
\rho(u(\tau), R)\leq \rho(u(\tau), A\tau^{\frac{1}{2-b}})\leq M_\infty^2(A, \tau_0),\quad \forall\ \tau\in[0, \tau_0].
\end{equation}
Thus, we derive that
\begin{equation}\label{f8}
\frac{1}{R^{2-b}}\int_{2R\leq|x|\leq4R}|u|^2dx\leq\frac{C}{R^{2-b-2s_c}}\rho(u(\tau), 2R)\leq\frac{C}{R^{2-b-2s_c}}M_\infty^2(A, \tau_0).
\end{equation}

In terms of the nonlinear term in \eqref{f4}, we apply the Gagliardo-Nirenberg inequality \eqref{f5} with $R=A\tau_0^{\frac{1}{2-b}}$ to obtain
\begin{equation}\label{f7}
\int_{|x|\geq R}|x|^c|u|^{p+2}dx\leq \eta\|\nabla u\|_{b,2}^2+\frac{C_\eta}{R^{2-b-2s_c}}M_\infty(A, \tau_0)^{\frac{2(4-np+2p)}{4-np}},
\end{equation}
where we have used \eqref{f6}. Now, plugging \eqref{f7} and \eqref{f8} into \eqref{f4} and choosing $\eta$ small enough, we obtain
\begin{eqnarray}\nonumber
\lefteqn{
V''_{\psi_R}(\tau)\leq 8(ps_c+2-b)E(u_0)-4ps_c\|\nabla u\|^2_{b,2}}\\\nonumber
&&\quad\quad+\frac{C}{R^{2-b-2s_c}}\left(M_\infty(A, \tau_0)^{\frac{2(4-np+2p)}{4-np}}+M^2_\infty(A, \tau_0)\right).\quad
\end{eqnarray}
For any $\tau\in[0, \tau_0]$, we integrate the above inequality in time from 0 to $\tau$ to get
\begin{eqnarray}\nonumber
\lefteqn{
4ps_c\int_0^\tau\|\nabla u(s)\|^2_{b,2}ds+V'_{\psi_R}(\tau)}\\\nonumber
&&\quad\quad\quad\leq V'_{\psi_R}(0)+8(ps_c+2-b)E(u_0)\tau\\\nonumber
&&\quad\quad\quad\quad+\frac{C\tau}{R^{2-b-2s_c}}\left(M_\infty(A, \tau_0)^{\frac{2(4-np+2p)}{4-np}}+M^2_\infty(A, \tau_0)\right).
\end{eqnarray}
Then we integrate one more time from 0 to $\tau_0$ and use \eqref{e5} to get
\begin{eqnarray}\nonumber
\lefteqn{
4ps_c\int_0^{\tau_0}\int_0^\tau\|\nabla u(s)\|^2_{b,2}dsd\tau+V_{\psi_R}(\tau_0)-V_{\psi_R}(0)}\\\nonumber
&&\quad\quad\quad\quad\leq V'_{\psi_R}(0)\tau_0+4(ps_c+2-b)E(u_0)\tau_0^2\\\label{f15}
&&\quad\quad\quad\quad\quad+\frac{C\tau_0^2}{2R^{2-b-2s_c}}\left(M_\infty(A, \tau_0)^{\frac{2(4-np+2p)}{4-np}}+M^2_\infty(A, \tau_0)\right).\quad
\end{eqnarray}
Now, similar to the proof of \eqref{d25}, we deduce that
\begin{eqnarray}\label{f16}
\int\psi_R|u_0|^2dx\leq C_b\int_{|x|\leq4R}|x|^{2-b}|u_0|^2dx\leq C_bR^{2-b+2s_c}\|u_0\|^2_{\gamma, \sigma_0},
\end{eqnarray}
where the last inequality is obtained by \eqref{f1}. Thus, since $R=A\tau_0^{\frac{1}{2-b}}$, we collect \eqref{f15} and \eqref{f16} to get
\begin{eqnarray}\nonumber
\lefteqn{
4ps_c\int_0^{\tau_0}(\tau_0-\tau)\|\nabla u(\tau)\|^2_{b,2}d\tau+V_{\psi_R}(\tau_0)}\\\nonumber
&&\leq CA^{2-b+2s_c}\tau_0^{1+\frac{2s_c}{2-b}}\|u_0\|^2_{\gamma, \sigma_0}+V'_{\psi_R}(0)\tau_0\\\nonumber
&&\ +4(ps_c+2-b)E(u_0)\tau_0^2+C\tau_0^{1+\frac{2s_c}{2-b}}\frac{M_\infty(A, \tau_0)^{\frac{2(4-np+2p)}{4-np}}+M^2_\infty(A, \tau_0)}{A^{2-b-2s_c}},
\end{eqnarray}
which concludes \eqref{f17}.

Next, by the definition of $\phi_R$ and $\nabla\psi_R=\frac{\nabla\phi_R}{|x|^b}$, we have
\begin{eqnarray}\nonumber
\lefteqn{\frac{1}{R^{2-b+2s_c}}V_{\psi_R}(\tau_0)\geq\frac{1}{R^{2-b+2s_c}}\int_{R\leq|x|\leq2R}|x|^{2-b}|u(\tau_0)|^2dx}\\\label{f19}
&&\quad\quad\quad\quad\quad\ \ \geq \frac{1}{R^{2s_c}}\int_{R\leq|x|\leq2R}|u(\tau_0)|^2dx.\quad\quad\quad\quad\quad\quad
\end{eqnarray}
We now divide \eqref{f15} by $R^{2-b+2s_c}$ with $R=A\tau_0^{\frac{1}{2-b}}$ and use \eqref{f19} to derive
\begin{eqnarray}\nonumber
\lefteqn{\frac{1}{R^{2s_c}}\int_{R\leq|x|\leq2R}|u(\tau_0)|^2dx}\\\nonumber
&&\leq \frac{1}{\tau_0^{\frac{2s_c}{2-b}}A^{2-b+2s_c}}\left(V'_{\psi_R}(0)+4(ps_c+2-b)E(u_0)\tau_0\right)\\\nonumber
&&\quad+C\|u_0\|^2_{\gamma, \sigma_0}+\frac{C}{A^{2(2-b)}}\left(M_\infty(A, \tau_0)^{\frac{2(4-np+2p)}{4-np}}+M^2_\infty(A, \tau_0)\right),
\end{eqnarray}
which is the desired result \eqref{f18}. We finish the proof of Lemma \ref{lem7}.
\end{proof}

Next, we prove a uniform control of the functional $\rho$ and a dispersive control of the solutions to \eqref{a0} in a parabolic region in time, assuming an initial control on the $L^{\sigma_0}(|x|^\gamma dx)$ norm and the energy.
\begin{pro}\label{pro4.1}
Let $n\geq3,\ 2-n<b\leq0,\ b-2<c<\frac{nbp}{4},\ 2(2-b)<\textbf{p}_c<(2-b)(p+2),\ p<\frac{4}{n}$ and $\sigma_0>1$ be defined as \eqref{f64} with $-n<\gamma\leq0$. Let $u\in C([0, \tau_\ast];\ \dot{H}^{s_c}\cap\dot{W}_b^{1,2})$ be the solution to \eqref{a0} with initial data $u_0\in\dot{H}^{s_c}\cap\dot{W}_b^{1,2}$. Assume
\begin{equation}\label{f20}
\tau_\ast^{1-\frac{2s_c}{2-b}}\max\{E(u_0),\ 0\}<1
\end{equation}
and
\begin{equation}\label{f21}
M_0:=\frac{4\|u_0\|_{\gamma, \sigma_0}}{\|\mathcal{Q}\|_{\gamma, \sigma_0}}>2,
\end{equation}
where $\mathcal{Q}$ is the solution to the elliptic equation \eqref{d16}. Then there exist universal constants $C_1,\ \alpha_1,\ \alpha_2>0$ depending only on $n,b,p,c$ such that for all $\tau_0\in[0, \tau_\ast]$, the following uniform control of the functional
\begin{equation}\label{f22}
\rho(u(\tau_0), M_0^{\alpha_1}\tau_0^{\frac{1}{2-b}})\leq C_1M_0^2
\end{equation}
and the global dispersive estimate holds
\begin{equation}\label{f23}
\int_0^{\tau_\ast}(\tau_\ast-\tau)\|\nabla u(\tau)\|_{b,2}^2d\tau\leq M_0^{\alpha_2}\tau_\ast^{1+\frac{2 s_c}{2-b}}.
\end{equation}

\end{pro}
\begin{proof}
We split the proof into three steps.

\textbf{Step 1.} Let $u\in C([0, \tau_\ast];\ \dot{H}^{s_c}\cap\dot{W}_b^{1,2})$ be a solution to \eqref{a0} with initial data $u_0\in\dot{H}^{s_c}\cap\dot{W}_b^{1,2}$ and $\epsilon>0$ be a fixed small enough real number to be chosen later. Define
\begin{equation}\label{f29}
G_\epsilon=M_0^{\frac{1}{\epsilon}},\quad\quad\quad A_\epsilon=\left(\frac{\epsilon G_\epsilon}{M_0^2}\right)^{\frac{1}{2-b+2s_c}},
\end{equation}
where $M_0$ is given in \eqref{f21}. Consider the following estimates
\begin{equation}\label{f24}
\int_0^{\tau_0}(\tau_0-\tau)\|\nabla u(\tau)\|_{b,2}^2d\tau\leq G_\epsilon \tau_0^{1+\frac{2 s_c}{2-b}}
\end{equation}
and
\begin{equation}\label{f25}
M^2_\infty(A_\epsilon, \tau_0)\leq\frac{2M_0^2}{\epsilon}.
\end{equation}
We claim that the largest time $\tau_1\in[0, \tau_\ast]$ such that both \eqref{f24} and \eqref{f25} hold for all $\tau_0\in[0, \tau_1]$ must exist.

Indeed, by the regularity of $u\in C([0, \tau_\ast];\ \dot{H}^{s_c}\cap\dot{W}_b^{1,2})$, there exists $0<\delta_1<\min\{1, \tau_\ast\}$ such that
\[\|\nabla u(\tau)\|_{b,2}^2<1+\|\nabla u_0\|_{b,2}^2,\quad\quad\forall\ \tau\in[0, \delta_1].\]
Multiplying the above inequality by $\tau_0-\tau$ and integrating from 0 to $\tau_0$ with $\tau_0\in[0, \delta_1]$, we have
\[\int_0^{\tau_0}(\tau_0-\tau)\|\nabla u(\tau)\|_{b,2}^2d\tau<\frac{1+\|\nabla u_0\|_{b,2}^2}{2}\tau_0^2\leq M_0^{\frac{1}{\epsilon}}\tau_0^{1+\frac{2 s_c}{2-b}},\]
for sufficiently small $\epsilon>0$ due to $M_0\geq2$ and $0<s_c<\frac{2-b}{2}$.

On the other hand, by Lemma \ref{lem5.1} (i), we obtain
\begin{equation}\label{f28}
\rho(u, R)\leq \rho(u, \frac{R}{2})\leq C\|u\|_{\gamma, \sigma_0}^2.
\end{equation}
Since $u\in C([0, \tau_\ast];\ \dot{H}^{s_c}\cap\dot{W}_b^{1,2})$, it follows from \eqref{f28} and the weighted Sobolev embedding $\dot{H}^{s_c}\hookrightarrow L^{\sigma_0}(|x|^\gamma dx)$ that
there exists $0<\delta_2\leq\tau_\ast$ such that for all $\tau\in[0, \delta_2]$,
\[\rho(u(\tau), A_\epsilon\tau^{\frac{1}{2-b}})\leq C\|u(\tau)\|_{\gamma, \sigma_0}^2\leq C(\epsilon_1+\|u_0\|_{\gamma, \sigma_0})^2\]
with sufficiently small $\epsilon_1>0$. Thus, from the definition \eqref{f27} and \eqref{f21}, 
\[M^2_\infty(A_\epsilon, \tau)\leq 2CM_0^2\leq\frac{2M_0^2}{\epsilon}\]
holds for sufficiently small $\epsilon>0$, which shows the existence of the largest time $\tau_1\in[0, \tau_\ast]$.

\textbf{Step 2.} The goal is to show that $\tau_1=\tau_\ast$, and therefore both the estimates \eqref{f24} and \eqref{f25} hold at the time $\tau_\ast$, which clearly yields \eqref{f22} and \eqref{f23}, in view of the definition \eqref{f29}.

To our aim, for any $\tau_0\in[0, \tau_1],$ we deduce from \eqref{f29}, \eqref{f25} and \eqref{f21} that
\begin{eqnarray}\nonumber
\lefteqn{\frac{M_\infty(A_\epsilon, \tau_0)^{\frac{2(4-np+2p)}{4-np}}+M^2_\infty(A_\epsilon, \tau_0)}{A_\epsilon^{2-b-2s_c}}\leq\frac{(\frac{2M_0^2}{\epsilon})^{\frac{4-np+2p}{4-np}}+\frac{2M_0^2}{\epsilon}}{(\frac{\epsilon G_\epsilon}{M_0^2})^{\frac{2-b-2s_c}{2-b+2s_c}}}}\\\nonumber
&&\quad\quad\quad\quad\quad\quad\quad\quad\quad\quad\quad\quad\quad\quad\leq(\frac{M_0}{\epsilon})^C\frac{1}{G_\epsilon^{\frac{2-b-2s_c}{2-b+2s_c}}}\leq\frac{1}{\epsilon^CM_0^{\frac{2-b-2s_c}{2-b+2s_c}\frac{1}{\epsilon}-C}}\leq\frac{1}{10},\quad
\end{eqnarray}
where $C>0$ depends only on $n, p, b, c$. So by Lemma \ref{lem7}, we inject this inequality into the estimate \eqref{f17} for $R=A_\epsilon\tau_0^{\frac{1}{2-b}}$ and use \eqref{b26}, \eqref{f21} to obtain
\begin{eqnarray}\nonumber
\lefteqn{\int_0^{\tau_0}(\tau_0-\tau)\|\nabla u(\tau)\|_{b,2}^2d\tau}\\\nonumber
&&\leq C\tau_0^{1+\frac{2s_c}{2-b}}(A_\epsilon^{2-b+2s_c}M_0^2+\frac{1}{10})+C\tau_0\left(V'_{\psi_R}(0)+4(ps_c+2-b)E(u_0)\tau_0\right)\\\nonumber
&&\leq G_\epsilon \tau_0^{1+\frac{2s_c}{2-b}}\left(\epsilon C+\frac{C}{10G_\epsilon}+\frac{C}{G_\epsilon\tau_0^{\frac{2s_c}{2-b}}}(V'_{\psi_R}(0)+4(ps_c+2-b) E(u_0)\tau_0)\right)\\\label{f31}
&&\leq G_\epsilon \tau_0^{1+\frac{2s_c}{2-b}}\left(\frac{1}{10}+\frac{C}{G_\epsilon\tau_0^{\frac{2s_c}{2-b}}}(V'_{\psi_R}(0)+4(ps_c+2-b) E(u_0))\tau_0\right)
\end{eqnarray}
for $\epsilon>0$ small enough. Now we claim that for any $\tau_0\in[0, \tau_1],\ A\geq A_\epsilon$ and $R=A\tau_0^{\frac{1}{2-b}}$, there exists a universal constant $C>0$ such that the following estimate holds
\begin{equation}\label{f30}
V'_{\psi_R}(0)+4(ps_c+2-b) E(u_0)\tau_0\leq \frac{CM_0^2A^{2-b+2s_c}}{\epsilon^{\frac{2-b}{2-b+2s_c}}}\tau_0^{\frac{2s_c}{2-b}}.
\end{equation}

We postpone the proof of \eqref{f30} to the next step. Assume that \eqref{f30} holds, we plug \eqref{f30} at $A=A_\epsilon$ into \eqref{f31} and use \eqref{f29} to get
\begin{eqnarray}\nonumber
\lefteqn{\int_0^{\tau_0}(\tau_0-\tau)\|\nabla u(\tau)\|_{b,2}^2d\tau\leq G_\epsilon \tau_0^{1+\frac{2s_c}{2-b}}\left(\frac{1}{10}+\frac{CM_0^2A_\epsilon^{2-b+2s_c}}{G_\epsilon\epsilon^{\frac{2-b}{2-b+2s_c}}}\right)}\\\label{f33}
&&\quad\quad\quad\quad\quad\quad\quad\quad\quad\ =G_\epsilon \tau_0^{1+\frac{2s_c}{2-b}}\left(\frac{1}{10}+C\epsilon^{\frac{2s_c}{2-b+2s_c}}\right)\leq\frac{G_\epsilon}{2}\tau_0^{1+\frac{2s_c}{2-b}},\quad\quad
\end{eqnarray}
provided $\epsilon>0$ has been chosen small enough.

Moreover, for $A\geq A_\epsilon$ and $R=A_\epsilon\tau_0^{\frac{1}{2-b}}$, since $\rho$ is non-increasing function in $R$, for any $\tau_0\in[0, \tau_1]$, we obtain
\begin{equation}\label{f32}
M^2_\infty(A, \tau_0)\leq M^2_\infty(A_\epsilon, \tau_0)\leq M^2_\infty(A_\epsilon, \tau_1)\leq\frac{2M_0^2}{\epsilon},
\end{equation}
where the last inequality is deduced by \eqref{f25}. Thus we invoke the estimate \eqref{f18} in Lemma \ref{lem7}, and collect \eqref{f30}, \eqref{f21}, \eqref{f32} to obtain
\begin{eqnarray}\nonumber
\lefteqn{\frac{1}{R^{2s_c}}\int_{R\leq|x|\leq2R}|u(\tau_0)|^2dx}\\\nonumber
&&\quad\quad\leq CM_0^2+\frac{C}{A^{2(2-b)}}\left((\frac{2M_0^2}{\epsilon})^{\frac{4-np+2p}{4-np}}+\frac{2M_0^2}{\epsilon}\right)+\frac{CM_0^2}{\epsilon^{\frac{2-b}{2-b+2s_c}}}\\\nonumber
&&\quad\quad\leq CM_0^2+\frac{C}{(\frac{\epsilon G_\epsilon}{M_0^2})^{\frac{2(2-b)}{2-b+2s_c}}}\left((\frac{2M_0^2}{\epsilon})^{\frac{4-np+2p}{4-np}}+\frac{2M_0^2}{\epsilon}\right)+\frac{CM_0^2}{\epsilon^{\frac{2-b}{2-b+2s_c}}}\\\nonumber
&&\quad\quad\leq\frac{M_0^2}{\epsilon}\left(C\epsilon+\frac{C}{\epsilon^{\widetilde{C}}M_0^{\frac{1}{\epsilon}\frac{2(2-b)}{2-b+2s_c}-2\widetilde{C}}}+C\epsilon^{\frac{2s_c}{2-b+2s_c}}\right)<\frac{M_0^2}{\epsilon}
\end{eqnarray}
for sufficiently small $\epsilon>0$. So we obtain
\begin{equation}\label{f34}
M^2_\infty(A_\epsilon, \tau_0)<\frac{M_0^2}{\epsilon}.
\end{equation}

Finally, using the regularity of $u$ again, the estimates \eqref{f33} and \eqref{f34} yield that $\tau_1=\tau_\ast$. Therefore, we conclude the estimates \eqref{f22} and \eqref{f23}.

\textbf{Step 3.} This step is devoted to the proof of the key estimate \eqref{f30}. For any $\tau_0\in[0, \tau_1]$, we first claim that there exists a universal constant $\widetilde{C}>0$ and a time $t_0$ such that
\begin{equation}\label{f35}
t_0\in[\frac{\tau_0}{4}\epsilon^{\frac{2-b}{2-b+2s_c}},\ \frac{\tau_0}{2}\epsilon^{\frac{2-b}{2-b+2s_c}}]\quad\quad\mbox{and}\quad\quad \|\nabla u(t_0)\|_{b,2}^2\leq\frac{\widetilde{C}G_\epsilon}{t_0^{1-\frac{2s_c}{2-b}}}.
\end{equation}

In fact, if \eqref{f35} does not hold, then, with $t=\epsilon^{\frac{2-b}{2-b+2s_c}}\tau_0\leq\tau_0\leq\tau_1,$ it follows from \eqref{f24} that
\begin{eqnarray}\nonumber
\lefteqn{G_\epsilon t^{1+\frac{2 s_c}{2-b}}\geq\int_0^{t}(t-\tau)\|\nabla u(\tau)\|_{b,2}^2d\tau\geq\frac{t}{2}\int_{\frac{t}{4}}^{\frac{t}{2}}\|\nabla u(\tau)\|_{b,2}^2d\tau}\\\nonumber
&&\quad\quad\ >\frac{t\widetilde{C}G_\epsilon}{2}\int_{\frac{t}{4}}^{\frac{t}{2}}\frac{1}{\tau^{1-\frac{2s_c}{2-b}}}d\tau\geq C\widetilde{C}G_\epsilon t^{1+\frac{2s_c}{2-b}},\quad
\end{eqnarray}
which implies a contradiction for $\widetilde{C}>0$ large enough. This shows that \eqref{f35} is valid.

Next, for $A\geq A_\epsilon$, if $A_1>0$ satisfies $R=A\tau_0^{\frac{1}{2-b}}=A_1t_0^{\frac{1}{2-b}}$ and
\begin{equation}\nonumber
\left(\frac{\epsilon^{\frac{2-b}{2-b+2s_c}}}{4}\right)^{\frac{1}{2-b}}\leq\frac{A}{A_1}\leq\epsilon^{\frac{1}{2-b+2s_c}},
\end{equation}
then we can estimate the integral term $V_{\psi_R}(t_0)$. From the virial identity \eqref{b26} and using the H\"{o}lder inequality, we obtain
\begin{equation}\label{f37}
V'_{\psi_R}(\tau)\leq C\|\nabla u\|_{b,2}(\int|\nabla\phi_R|^2|x|^{-b}|u|^2dx)^{\frac{1}{2}}.
\end{equation}
By the properties of $\theta$ in \eqref{f36} and the relation with $\nabla\psi_R=\frac{\nabla\phi_R}{|x|^b}$, one can easily deduce that
\[|\nabla\phi_R|^2|x|^{-b}=|R\theta'(\frac{r}{R})|^2|x|^{-b}\leq C\psi_R.\]
Thus, combining the last inequality with \eqref{f37}, we have
\[\left(V_{\psi_R}(\tau)^{\frac{1}{2}}\right)'\leq C\|\nabla u\|_{b,2}.\]
Now, we integrate the above differential inequality from 0 to $t_0$ and square both sides to obtain
\begin{eqnarray}\nonumber
\lefteqn{V_{\psi_R}(t_0)\leq C\left(V_{\psi_R}(0)+(\int_0^{t_0}\|\nabla u(\tau)\|_{b,2}d\tau)^2\right)}\\\nonumber
&&\quad\ \leq C\left(R^{2-b+2s_c}M_0^2+t_0\int_0^{t_0}\|\nabla u(\tau)\|^2_{b,2}d\tau\right)\\\nonumber
&&\quad\ \leq C\left(R^{2-b+2s_c}M_0^2+\int_0^{2t_0}(2t_0-\tau)\|\nabla u(\tau)\|^2_{b,2}d\tau\right),
\end{eqnarray}
where the last but one inequality is obtained by \eqref{f16} and \eqref{f21}.

Since $2t_0\leq \epsilon^{\frac{2-b}{2-b+2s_c}}\tau_0\leq \tau_0\leq\tau_1$ and $R=A_1t_0^{\frac{1}{2-b}}$, we link \eqref{f24} with the above inequality to obtain
\begin{eqnarray}\label{f38}
V_{\psi_R}(t_0)\leq C(R^{2-b+2s_c}M_0^2+G_\epsilon t_0^{1+\frac{2 s_c}{2-b}})\leq C R^{2-b+2s_c}\left(M_0^2+\frac{G_\epsilon}{A_1^{2-b+2s_c}}\right).
\end{eqnarray}
Therefore, the control of $V_{\psi_R}(t_0)$ is achieved.

Now, for $R=A\tau_0^{\frac{1}{2-b}}=A_1t_0^{\frac{1}{2-b}}$, and $A\geq A_\epsilon$, we deduce from \eqref{f35}, \eqref{f37} and \eqref{f38} that
\begin{eqnarray}\nonumber
\lefteqn{V'_{\psi_R}(t_0)\leq CR^{\frac{2-b+2s_c}{2}}\|\nabla u(t_0)\|_{b,2}(\frac{1}{R^{2-b+2s_c}}\int\psi_R|u(t_0)|^2dx)^{\frac{1}{2}}}\\\nonumber
&&\quad\ \leq CR^{\frac{2-b+2s_c}{2}}\frac{G_\epsilon^{\frac{1}{2}}}{t_0^{\frac{1}{2}-\frac{s_c}{2-b}}}\left(M_0^2+\frac{G_\epsilon}{A_1^{2-b+2s_c}}\right)^{\frac{1}{2}}\\\nonumber
&&\quad\ \leq CA^{\frac{2-b+2s_c}{2}}\tau_0^{\frac{1}{2}+\frac{s_c}{2-b}}\frac{G_\epsilon^{\frac{1}{2}}}{(\tau_0\frac{A^{2-b}}{A_1^{2-b}})^{\frac{1}{2}-\frac{s_c}{2-b}}}\left(M_0+\frac{G_\epsilon^{\frac{1}{2}}}{A_1^{\frac{2-b+2s_c}{2}}}\right),
\end{eqnarray}
which can be further estimated by
\begin{eqnarray}\nonumber
\lefteqn{V'_{\psi_R}(t_0)\leq CM_0^2A^{2-b+2s_c}\tau_0^{\frac{2s_c}{2-b}}\left((\frac{G_\epsilon}{A^{2-b+2s_c}M_0^2})^{\frac{1}{2}}(\frac{A_1}{A})^{\frac{2-b}{2}-s_c}+(\frac{G_\epsilon}{A^{2-b+2s_c}M_0^2})(\frac{A}{A_1})^{2s_c}\right)}\\\nonumber
&&\quad\ \ \leq CM_0^2A^{2-b+2s_c}\tau_0^{\frac{2s_c}{2-b}}\left((\frac{G_\epsilon}{A_\epsilon^{2-b+2s_c}M_0^2})^{\frac{1}{2}}(\frac{A_1}{A})^{\frac{2-b}{2}-s_c}+(\frac{G_\epsilon}{A_\epsilon^{2-b+2s_c}M_0^2})(\frac{A}{A_1})^{2s_c}\right)\\\label{f40}
&&\quad\ \ \leq \frac{CM_0^2A^{2-b+2s_c}}{\epsilon^{\frac{2-b}{2-b+2s_c}}}\tau_0^{\frac{2s_c}{2-b}}.
\end{eqnarray}
This shows that  the virial quantity $V'_{\psi_R}(t_0)$ at time $t_0$ is controlled.

Finally, we turn to the control on the virial quantity $V'_{\psi_R}(t_0)$ backwards at $\tau=0$. To do it, we use \eqref{b11} with $R=A\tau_0^{\frac{1}{2-b}}$ for all $\tau_0\in[0, \tau_1]$ and estimate the terms on the right-hand side.
Since
\begin{equation}\nonumber
|\Delta^2\phi_R|\lesssim R^{-2}\quad\quad \mbox{and}\quad\quad |\nabla\Delta\phi_R\cdot\nabla|x|^b|\lesssim R^{b-2},
\end{equation}
so for all $A\geq A_\epsilon$, we deduce from \eqref{f8} and \eqref{f25} that
\begin{eqnarray}\nonumber
\lefteqn{\left|\int(|x|^b\Delta^2\phi_R+\nabla\Delta\phi_R\cdot\nabla|x|^b)|u(\tau_0)|^2dx\right|}\\\nonumber
&&\quad\quad\quad\quad\quad\leq\frac{C}{R^{2-b}}\int_{2R\leq|x|\leq4R}|u(\tau_0)|^2dx\leq\frac{C M_\infty^2(A_\epsilon, \tau_0)}{R^{2-b-2s_c}}\\\nonumber
&&\quad\quad\quad\quad\quad\leq\frac{CM_0^2}{\epsilon A^{2-b-2s_c}\tau_0^{1-\frac{2s_c}{2-b}}}\leq\frac{CM_0^2\epsilon^{\frac{2-b-2s_c}{2-b+2s_c}}}{\epsilon A_\epsilon^{2-b-2s_c}t_0^{1-\frac{2s_c}{2-b}}},
\end{eqnarray}
where we used the relation $2t_0\leq\epsilon^{\frac{2-b}{2-b+2s_c}}\tau_0$ in the last inequality. So for $\epsilon>0$ small enough, we have
\begin{eqnarray}\label{f39}
\left|\int(|x|^b\Delta^2\phi_R+\nabla\Delta\phi_R\cdot\nabla|x|^b)|u(\tau_0)|^2dx\right|\leq\frac{CM_0^2}{\epsilon t_0^{1-\frac{2s_c}{2-b}}}\left(\frac{M_0^2}{G_\epsilon}\right)^{\frac{2-b-2s_c}{2-b+2s_c}}\leq\frac{1}{t_0^{1-\frac{2s_c}{2-b}}}.\quad
\end{eqnarray}
Therefore, putting \eqref{a10}, \eqref{b26}, \eqref{f39} all together, and using the conservation of the energy, we obtain
\begin{equation}\nonumber
|V''_{\psi_R}(t)|\leq C\left(\|\nabla u\|_{b,2}^2+|E(u_0)|+\frac{1}{t_0^{1-\frac{2s_c}{2-b}}}\right).
\end{equation}

We integrate this inequality in time from 0 to $t_0$, then it follows from \eqref{f40} and the estimate
\[\int_0^{t_0}\|\nabla u(\tau)\|^2_{b,2}d\tau\leq G_\epsilon t_0^{\frac{2 s_c}{2-b}}\]
in \eqref{f38} that
\begin{eqnarray}\nonumber
\lefteqn{|V'_{\psi_R}(0)|\leq |V'_{\psi_R}(t_0)|+C(\int_0^{t_0} \|\nabla u(\tau)\|_{b,2}^2d\tau+|E(u_0)|t_0+t_0^{\frac{2s_c}{2-b}})}\\\nonumber
&&\quad\quad\leq\frac{CM_0^2A^{2-b+2s_c}}{\epsilon^{\frac{2-b}{2-b+2s_c}}}\tau_0^{\frac{2s_c}{2-b}}+CG_\epsilon t_0^{\frac{2s_c}{2-b}}+C|E(u_0)|t_0\\\label{f41}
&&\quad\quad\leq\frac{CM_0^2A^{2-b+2s_c}}{\epsilon^{\frac{2-b}{2-b+2s_c}}}\tau_0^{\frac{2s_c}{2-b}}+C|E(u_0)|\epsilon^{\frac{2-b}{2-b+2s_c}}\tau_0,\quad\quad\quad
\end{eqnarray}
where we used \eqref{f29} and \eqref{f35} in the last inequality for all $A\geq A_\epsilon$. Notice that
\[|E(u_0)|\epsilon^{\frac{2-b}{2-b+2s_c}}+4(ps_c+2-b) E(u_0)\leq C\max\{E(u_0),\ 0\},\]
which together with \eqref{f41} yields
\begin{eqnarray}\nonumber
\lefteqn{V'_{\psi_R}(0)+4(ps_c+2-b)E(u_0)\tau_0}\\\nonumber
&&\quad\leq\frac{CM_0^2A^{2-b+2s_c}}{\epsilon^{\frac{2-b}{2-b+2s_c}}}\tau_0^{\frac{2s_c}{2-b}}+C\left(|E(u_0)|\epsilon^{\frac{2-b}{2-b+2s_c}}+4(ps_c+2-b) E(u_0)\right)\tau_0\\\nonumber
&&\quad\leq\frac{CM_0^2A^{2-b+2s_c}}{\epsilon^{\frac{2-b}{2-b+2s_c}}}\tau_0^{\frac{2s_c}{2-b}}+C\max\{E(u_0),\ 0\}\tau_0
\end{eqnarray}
for $\epsilon>0$ small enough. Hence, by hypothesis \eqref{f20}, we conclude that
\[V'_{\psi_R}(0)+4(ps_c+2-b)E(u_0)\tau_0\leq\frac{CM_0^2A^{2-b+2s_c}}{\epsilon^{\frac{2-b}{2-b+2s_c}}}\tau_0^{\frac{2s_c}{2-b}}.\]
The estimate \eqref{f30} is proved. We complete the proof of Proposition \ref{pro4.1}.

\end{proof}

Under the assumption of Proposition \ref{pro4.1} and an additional energetic constraint on the solution, we prove the lower bound on a weighted local $L^2$ norm of the initial data $u_0$.

\begin{pro}\label{pro4.2}
Let $n\geq3,\ 2-n<b\leq0,\ b-2<c<\frac{nbp}{4},\ 2(2-b)<\textbf{p}_c<(2-b)(p+2),\ p<\frac{4}{n}$ and $\sigma_0$ be defined as \eqref{f64} with $-n<\gamma\leq0$. Let $u\in C([0, \tau_\ast];\ \dot{H}^{s_c}\cap\dot{W}_b^{1,2})$ be a solution to \eqref{a0} with initial data $u_0\in\dot{H}^{s_c}\cap\dot{W}_b^{1,2}$ such
that the conditions \eqref{f20} and \eqref{f21} of Proposition \ref{pro4.1} hold. Let $\tau_0\in[0, \frac{\tau_\ast}{2}]$ and define $\lambda_u(\tau_0)=\|\nabla u\|_{b,2}^{-\frac{2-b}{2-b-2s_c}}$, we suppose
\begin{equation}\label{f42}
E(u_0)\leq\frac{1}{4\lambda_u^{2-\frac{4s_c}{2-b}}(\tau_0)}.
\end{equation}
Then there exist universal constants $C_2,\ \alpha_3>0$ depending on $n, b, p$, such that if
\begin{equation}\label{f43}
F_\ast=\frac{\tau_0^{\frac{1}{2-b}}}{\lambda_u^{\frac{2}{2-b}}(\tau_0)}\quad\quad\quad \mbox{and}\quad\quad D_\ast=M_0^{\alpha_3}\max\{1, F_\ast^{\frac{2-b+2s_c}{2-b-2s_c}}\},
\end{equation}
we have
\begin{equation}\label{f44}
\frac{1}{\lambda_u^{\frac{4s_c}{2-b}}(\tau_0)}\int_{|x|\leq D_\ast\lambda_u^{\frac{2}{2-b}}(\tau_0)}|u_0|^2dx\geq C_2.
\end{equation}

\end{pro}

\begin{proof}
\textbf{Step 1.} We first claim that there exists a universal constant $C_2>0$ such that
\begin{equation}\label{f45}
\frac{1}{\lambda_u^{\frac{4s_c}{2-b}}(\tau_0)}\int_{|x|\leq 2J_\ast\lambda_u^{\frac{2}{2-b}}(\tau_0)}|u(\tau_0)|^2dx\geq 2C_2
\end{equation}
with
\begin{equation}\label{f46}
J_\ast=C_\ast\max\{M_0^{\alpha_1}F_\ast,\ M_0^{\frac{2(4-np+2p)}{(4-np)(2-b-2s_c)}}\}
\end{equation}
for $C_\ast>1$ large enough to be chosen later.

In fact, we define a renormalization of $u(\tau_0)$ as
\begin{equation}\label{f47}
\omega(x)=\lambda_u(\tau_0)^{\frac{2(2-b+c)}{(2-b)p}}u(\lambda_u^{\frac{2}{2-b}}(\tau_0)x, \tau_0),
\end{equation}
where $\lambda_u(\tau_0)=\|\nabla u(\tau_0)\|_{b,2}^{-\frac{2-b}{2-b-2s_c}}$. Thus it follows from \eqref{f42} and the energy conservation that
\begin{equation}\nonumber
\|\nabla\omega\|_{b,2}=1\quad\quad \mbox{and}\quad\quad E(\omega)=\lambda_u^{2-\frac{4s_c}{2-b}}(\tau_0)E(u(\tau_0))\leq\frac{1}{4}.
\end{equation}
So we obtain
\begin{equation}\label{f48}
\int|x|^c|\omega|^{p+2}dx=(p+2)(\frac{1}{2}\|\nabla\omega\|_{b,2}^2-E(\omega))\geq\frac{p+2}{4}.
\end{equation}

Denote $J_\ast$ as in \eqref{f46}. By the definition of the semi-norm $\rho$ and \eqref{f47}, we obtain
\begin{eqnarray}\nonumber
\lefteqn{\rho(\omega, J_\ast)\leq\sup_{R\geq\frac{J_\ast}{C_\ast}}\frac{1}{R^{2s_c}}\int_{R\leq|x|\leq2R}|\omega|^2dx\leq\sup_{R\geq M_0^{\alpha_1}F_\ast}\frac{1}{R^{2s_c}}\int_{R\leq|x|\leq2R}|\omega|^2dx}\\\nonumber
&&\quad\ \ =\sup_{R\geq M_0^{\alpha_1}F_\ast}\frac{1}{(\lambda_u^{\frac{2}{2-b}}(\tau_0)R)^{2s_c}}\int_{\lambda_u^{\frac{2}{2-b}}(\tau_0)R\leq|x|\leq2\lambda_u^{\frac{2}{2-b}}(\tau_0)R}|u(\tau_0)|^2dx\\\nonumber
&&\quad\ \ =\sup_{R\geq \lambda_u^{\frac{2}{2-b}}(\tau_0)M_0^{\alpha_1}F_\ast}\frac{1}{R^{2s_c}}\int_{R\leq|x|\leq2R}|u(\tau_0)|^2dx\\\nonumber
&&\quad\ \ =\rho(u(\tau_0), M_0^{\alpha_1}\tau_0^{\frac{1}{2-b}})\leq C_1M_0^2,
\end{eqnarray}
where the last two steps are obtained by the definition of $F_\ast$ and the estimate \eqref{f22} in Proposition \ref{pro4.1}.
Thus, injecting the above inequality to the estimate \eqref{f5} in Lemma \ref{lem5.1}, we further have
\begin{eqnarray}\nonumber
\lefteqn{\int_{|x|\geq J_\ast}|x|^c|\omega|^{p+2}dx\leq \eta\|\nabla \omega\|_{b,2}^2+\frac{C_\eta}{J_\ast^{2-b-2s_c}}\rho(\omega, J_\ast)^{\frac{4-np+2p}{4-np}}}\\\label{f49}
&&\quad\quad\quad\quad\quad\quad\ \ \leq \eta+\frac{C_\eta}{J_\ast^{2-b-2s_c}}M_0^{\frac{2(4-np+2p)}{4-np}}\leq2\eta\quad\quad\quad\quad
\end{eqnarray}
for $C_\ast>1$ large enough. Therefore, we deduce from \eqref{f48} and \eqref{f49} that
\begin{equation}\label{f52}
\int_{|x|\leq J_\ast}|x|^c|\omega|^{p+2}dx\geq\frac{p+2}{8}.
\end{equation}

Now, we rewrite the Gagliardo-Nirenberg inequality \eqref{f63} as
\begin{equation}\label{f51}
\int|x|^c|f|^{p+2}dx\leq C\|\nabla f\|_{b,2}^{2+\frac{2ps_c}{2-b}}\|f\|_2^{p-\frac{2ps_c}{2-b}},
\end{equation}
which can be used to deduce a lower bound \eqref{f45} on the $L^2$ norm of $u(\tau_0)$ around the origin. More precisely, Let $\Theta\in C_0^\infty(\Bbb{R}^n)$ be a cut-off function such that
\begin{equation}\Theta(x)=
\left\{
\begin{aligned}
 & 1,\quad\ \mbox{if}\ |x|\leq 1,\\\label{f53}
 & 0,\quad\ \mbox{if}\ |x|\geq2,
\end{aligned}\right.  \quad\quad \mbox{and}\quad\quad \Theta_A(x)=\Theta\left(\frac{x}{A}\right).
\end{equation}
Then one can easily see that
\begin{eqnarray}\nonumber
\lefteqn{\|\nabla(\Theta_{J_\ast}\omega)\|_{b,2}\leq\|\omega\nabla\Theta_{J_\ast}\|_{b,2}+\|\Theta_{J_\ast}\nabla\omega\|_{b,2}}\\\nonumber
&&\quad\quad\quad\quad\ \leq C\left((\int_{|x|\leq2J_\ast}|\omega|^2dx)^{\frac{1}{2}}+1\right).
\end{eqnarray}
Thus, combining the above inequality with \eqref{f51}, we have
\begin{eqnarray}\nonumber
\lefteqn{\int_{|x|\geq J_\ast}|x|^c|\omega|^{p+2}dx\leq\int|x|^c|\Theta_{J_\ast}\omega|^{p+2}dx\leq C\|\nabla(\Theta_{J_\ast}\omega)\|_{b,2}^{2+\frac{2ps_c}{2-b}}\|\Theta_{J_\ast}\omega\|_2^{p-\frac{2ps_c}{2-b}}}\\\nonumber
&&\quad\quad\quad\quad\quad\quad\ \ \leq C(\int_{|x|\leq2J_\ast}|\omega|^2dx)^{\frac{p+2}{2}}+C(\int_{|x|\leq2J_\ast}|\omega|^2dx)^{\frac{p(2-b-2s_c)}{2(2-b)}}.\quad\quad\quad\quad
\end{eqnarray}
It follows from \eqref{f52}  that
\begin{equation}\nonumber
0<2C_2\leq\int_{|x|\leq2J_\ast}|\omega|^2dx=\frac{1}{\lambda_u^{\frac{4s_c}{2-b}}(\tau_0)}\int_{|x|\leq 2J_\ast\lambda_u^{\frac{2}{2-b}}(\tau_0)}|u(\tau_0)|^2dx
\end{equation}
for some constant $C_2>0$ depending on $b$ and $J_\ast$. This arrives to \eqref{f45}.

\textbf{Step 2.} The estimate \eqref{f45} can also be verified at the initial time $\tau=0$. To do it, we set
\begin{equation}\nonumber
D_\epsilon=C_\epsilon\max\{F_\ast,\ F_\ast^{\frac{2-b+2s_c}{2-b-2s_c}}\}\max\{M_0^{\alpha_1},\ M_0^{\frac{2+\alpha_2}{2-b-2s_c}}\}
\end{equation}
and for all $D\geq D_\epsilon$,
\begin{equation}\label{f54}
\widetilde{R}=\widetilde{R}(D, \tau_0)=D\lambda_u^{\frac{2}{2-b}}(\tau_0),
\end{equation}
where $\alpha_2>0$ is chosen to satisfy $D_\epsilon\geq J_\ast$. Then we claim that for any $\epsilon>0$, there exists $C_\epsilon>0$ such that the following inequality
\begin{equation}\label{f55}
\frac{1}{\lambda_u^{\frac{4s_c}{2-b}}(\tau_0)}\left|\int\Theta_{\widetilde{R}}|u(\tau_0)|^2dx-\int\Theta_{\widetilde{R}}|u_0|^2dx\right|<\epsilon
\end{equation}
holds, where $\Theta_{\widetilde{R}}$ is defined as \eqref{f53}.

Indeed, we invoke the virial quantity \eqref{b26} with $\psi_R=\Theta_{\widetilde{R}}$ to obtain
\begin{eqnarray}\nonumber
\lefteqn{|V'_{\Theta_{\widetilde{R}}}(\tau)|=2|\textmd{Im}\int|x|^b\nabla\Theta_{\widetilde{R}}\cdot\nabla u(\tau)\overline{u}(\tau)dx|}\\\nonumber
&&\quad\quad\ \leq\frac{C}{\widetilde{R}^{\frac{2-b}{2}}}\|\nabla u(\tau)\|_{b,2}(\int_{\widetilde{R}\leq|x|\leq2\widetilde{R}}|u(\tau)|^2dx)^{\frac{1}{2}}\\\label{f56}
&&\quad\quad\ \leq\frac{C}{\widetilde{R}^{\frac{2-b}{2}-s_c}}\|\nabla u(\tau)\|_{b,2}\rho(u(\tau), \widetilde{R})^{\frac{1}{2}}.
\end{eqnarray}
Given $\epsilon>0$, for all $C_\epsilon>1$, it follows from \eqref{f43}, \eqref{f54} and the definition of $D_\epsilon$ that
\begin{equation}\nonumber
\widetilde{R}\geq\frac{D_\epsilon}{F_\ast}\tau_0^{\frac{1}{2-b}}\geq\frac{D_\epsilon}{F_\ast}\tau^{\frac{1}{2-b}}\geq M_0^{\alpha_1}\tau^{\frac{1}{2-b}},\quad\quad \forall\ \tau\in[0, \tau_0],
\end{equation}
which together with \eqref{f22} yields
\begin{equation}\label{f57}
\rho(u(\tau), \widetilde{R})\leq \rho(u(\tau), M_0^{\alpha_1}\tau^{\frac{1}{2-b}})\leq C_1M_0^2, \quad\quad \forall\ \tau\in[0, \tau_0].
\end{equation}
So, combining \eqref{f56} with \eqref{f57}, we obtain
\[|V'_{\Theta_{\widetilde{R}}}(\tau)|\leq\frac{CM_0}{\widetilde{R}^{\frac{2-b}{2}-s_c}}\|\nabla u(\tau)\|_{b,2}.\]
Now, we integrate this inequality from 0 to $\tau_0\in[0, \frac{\tau_\ast}{2}]$, divide the resulting inequality by $\lambda_u^{\frac{4s_c}{2-b}}(\tau_0)$ to get
\begin{eqnarray}\nonumber
\lefteqn{\frac{1}{\lambda_u^{\frac{4s_c}{2-b}}(\tau_0)}\left|V_{\Theta_{\widetilde{R}}}(\tau_0)-V_{\Theta_{\widetilde{R}}}(0)\right|\leq(\frac{D}{\widetilde{R}})^{2s_c}\frac{CM_0}{\widetilde{R}^{\frac{2-b}{2}-s_c}}\int_0^{\tau_0}\|\nabla u(\tau)\|_{b,2}d\tau}\\\nonumber
&&\quad\quad\quad\quad\quad\quad\quad\quad\quad\quad\ \ \leq\frac{CD^{2s_c}M_0}{\widetilde{R}^{\frac{2-b}{2}+s_c}}(\tau_0\int_0^{\tau_0}\|\nabla u(\tau)\|^2_{b,2}d\tau)^{\frac{1}{2}}\\\nonumber
&&\quad\quad\quad\quad\quad\quad\quad\quad\quad\quad\ \ \leq\frac{CD^{2s_c}M_0}{\widetilde{R}^{\frac{2-b}{2}+s_c}}(\int_0^{2\tau_0}(2\tau_0-\tau)\|\nabla u(\tau)\|^2_{b,2}d\tau)^{\frac{1}{2}}.
\end{eqnarray}
Thus, using the estimate \eqref{f23} in Proposition \ref{pro4.1}, we deduce that
\begin{eqnarray}\nonumber
\lefteqn{\frac{1}{\lambda_u^{\frac{4s_c}{2-b}}(\tau_0)}\left|V_{\Theta_{\widetilde{R}}}(\tau_0)-V_{\Theta_{\widetilde{R}}}(0)\right|\leq CM_0^{1+\frac{\alpha_2}{2}}D^{2s_c}\left(\frac{\tau_0^{\frac{1}{2-b}}}{\widetilde{R}}\right)^{\frac{2-b}{2}+s_c}}\\\nonumber
&&\quad\quad\quad\quad\quad\quad\quad\quad\quad\quad\quad=CM_0^{1+\frac{\alpha_2}{2}}\frac{F_\ast^{\frac{2-b}{2}+s_c}}{D^{\frac{2-b}{2}-s_c}},\quad\quad\quad\quad\quad
\end{eqnarray}
where we used the definitions \eqref{f43} and \eqref{f54} in the last identity. Thus \eqref{f55} follows from the definition of $D_\epsilon$ for all $D\geq D_\epsilon$ and $C_\epsilon>1$ large enough.

Finally, we are able to give the proof of \eqref{f44}. Putting \eqref{f53}, \eqref{f55} all together, we have
\begin{equation}\label{f72}
\frac{1}{\lambda_u^{\frac{4s_c}{2-b}}(\tau_0)}\int_{|x|\leq 2\widetilde{R}}|u_0|^2dx\geq\frac{1}{\lambda_u^{\frac{4s_c}{2-b}}(\tau_0)}\int\Theta_{\widetilde{R}}|u(\tau_0)|^2dx-\epsilon
\end{equation}
for $\epsilon>0$ small enough. Thus, by choosing $\alpha_3>0$ such that
\[M_0^{\alpha_3}\max\{1, F_\ast^{\frac{2-b+2s_c}{2-b-2s_c}}\}\geq2D_\epsilon,\]
we use the definition of $D_\ast$ and \eqref{f45}, \eqref{f72} to get
\[\frac{1}{\lambda_u^{\frac{4s_c}{2-b}}(\tau_0)}\int_{|x|\leq D_\ast\lambda_u^{\frac{2}{2-b}}(\tau_0)}|u_0|^2dx\geq\frac{1}{\lambda_u^{\frac{4s_c}{2-b}}(\tau_0)}\int_{|x|\leq 2D_\epsilon\lambda_u^{\frac{2}{2-b}}(\tau_0)}|u(\tau_0)|^2dx\geq C_2>0\]
for all $D_\epsilon\geq J_\ast$. Hence, we conclude \eqref{f44}. The proof of Proposition \ref{pro4.2} is completed.

\end{proof}

\section{Blow up in $\dot{H}^{s_c}\cap\dot{W}^{1,2}_b$ }
\subsection{Proof of Theorem \ref{thm2}.}
In this part, we give the proof of Theorem \ref{thm2}, which follows as a direct consequence of Propositions \ref{pro4.1} and \ref{pro4.2}.

\textbf{Proof of Theorem \ref{thm2}.} Let $u_0\in\dot{H}^{s_c}\cap\dot{W}^{1,2}_b$ and $E(u_0)\leq0$. We argue by contradiction and assume that there exists a global solution $u\in C([0, +\infty);\ \dot{H}^{s_c}\cap\dot{W}^{1,2}_b)$ to the dINLS. By Lemma \ref{pro1} and the energy conservation, we have
\begin{equation}\label{f80}
E(u_0)\geq \frac{1}{2}\|\nabla u\|_{b,2}^2\left(1-\frac{\|u\|^p_{\gamma, \sigma_0}}{\|\mathcal{Q}\|^p_{\gamma, \sigma_0}}\right),
\end{equation}
which together with $E(u_0)\leq0$ yields that $u_0$ satisfies the conditions \eqref{f20} and \eqref{f21}. Thus, we apply Proposition \ref{pro4.1} to obtain
\[\int_0^{\tau_\ast}(\tau_\ast-\tau)\|\nabla u(\tau)\|_{b,2}^2d\tau\leq C_{u_0}\tau_\ast^{1+\frac{2s_c}{2-b}},\quad\quad \forall\ \tau_\ast>0.\]
Therefore, there exists a sequence $\{\tau_n\}_{n\in\Bbb{N}}$ such that
\begin{equation}\label{f61}
\|\nabla u(\tau_n)\|_{b,2}\leq\frac{C_{u_0}}{\tau_n^{\frac{2-b-2s_c}{2(2-b)}}},\quad\quad \forall\ n\in\Bbb{N}.
\end{equation}
Since $\lambda_u(\tau)=\|\nabla u(\tau)\|_{b,2}^{-\frac{2-b}{2-b-2s_c}}$, we further have
\begin{equation}\nonumber
\lambda_u(\tau)\geq C_{u_0}\sqrt{\tau_n}.
\end{equation}

Now, setting $\tau_0=\tau_n$ in Proposition \ref{pro4.2}, we deduce from the above inequality that
\[F_\ast=F_n=\frac{\tau_n^{\frac{1}{2-b}}}{\lambda_u^{\frac{2}{2-b}}(\tau_n)}\leq C_{u_0}.\]
So, the conditions \eqref{f42} and \eqref{f43} are all satisfied from $E(u_0)\leq0$. Therefore, by Proposition \ref{pro4.2}, there exists a $D_\ast$ independent of $n$ such that
\begin{equation}\label{f60}
\frac{1}{\lambda_u^{\frac{4s_c}{2-b}}(\tau_0)}\int_{|x|\leq D_\ast\lambda_u^{\frac{2}{2-b}}(\tau_0)}|u_0|^2dx\geq C_2>0.
\end{equation}

On the other hand,  since $\lambda_u(\tau_n)\rightarrow+\infty$ as $n\rightarrow+\infty$ due to \eqref{f61}, by picking $R=D_\ast\lambda_u^{\frac{2}{2-b}}(\tau_n)$ in Lemma \ref{lem5.1}, we get
\begin{equation}\nonumber
\lim_{n\rightarrow+\infty}\frac{1}{(D_\ast\lambda_u^{\frac{2}{2-b}}(\tau_n))^{2s_c}}\int_{|x|\leq D_\ast\lambda_u^{\frac{2}{2-b}}(\tau_n)}|u_0|^2dx=0,
\end{equation}
which contradicts with \eqref{f60}. Hence, the maximal existence time $T^\ast$ of the solution to \eqref{a0} is finite. We conclude the proof of Theorem \ref{thm2}.

\subsection{Lower bound for the blow-up rate}
With Propositions \ref{pro4.1} and \ref{pro4.2} in hand, we are able to give the proof of Theorem \ref{thm3}.

\textbf{Proof of Theorem \ref{thm3}.}
Given $u_0\in \dot{H}^{s_c}\cap \dot{W}_b^{1,2}$, let $u$ be the corresponding finite time blow-up solutions to \eqref{a0} with $0<T^\ast<+\infty$. We deduce from \eqref{f92} that
\begin{equation}\label{f77}
\lambda_u(t)\leq C\sqrt{T^\ast-t},
\end{equation}
where $\lambda_u(t)=\|\nabla u(t)\|_{b,2}^{-\frac{2-b}{2-b-2s_c}}$ is defined as in Proposition \ref{pro4.2}.

For $t$ close enough to $T^\ast$, we consider the following renormalization of $u$:
\begin{equation}\label{f84}
v^{(t)}(x, \tau)=\lambda_u(t)^{\frac{2(2-b+c)}{(2-b)p}}\overline{u}(\lambda^{\frac{2}{2-b}}_u(t)x,\ t-\lambda^2_u(t)\tau)
\end{equation}
and write $v=v^{(t)}$ for simplicity. By a straightforward calculation, we obtain
\begin{equation}\label{f76}
\|v(0)\|_{\gamma, \sigma_0}=\lambda_u(t)^{\frac{2(2-b+c)}{(2-b)p}}(\int|x|^\gamma|u(\lambda^{\frac{2}{2-b}}_u(t)x,\ t)|^{\sigma_0}dx)^{\frac{1}{\sigma_0}}=\|u(t)\|_{\gamma, \sigma_0}
\end{equation}
and
\begin{equation}\label{f79}
\|\nabla v(0)\|_{b,2}=1\quad\quad \mbox{and}\quad \quad E(v(0))=\lambda^{2-\frac{4s_c}{2-b}}_u(t)E(u_0).
\end{equation}
Moreover, from the scaling symmetry $u(x, t)\mapsto\lambda^{\frac{2(2-b+c)}{(2-b)p}}u(\lambda^{\frac{2}{2-b}}x, \lambda^2t)$, one can easily check that $v: \tau\mapsto v(\tau)$ is the solution to \eqref{a0} on the time interval $[0, \frac{t}{\lambda^2_u(t)}]$.

To prove \eqref{f75}, if we set
\begin{equation}\label{}
N(t)=-\log\lambda_u(t),\quad\mbox{or\ equivalently}\quad e^{N(t)}=\frac{1}{\lambda_u(t)},
\end{equation}
then it follows from \eqref{f77} that
\[N(t)\geq C|\log(T^\ast-t)|.\]
So, together with \eqref{f76}, it suffices to prove
\begin{equation}\label{f78}
\|v(0)\|_{\gamma, \sigma_0}\geq N(t)^l
\end{equation}
for some universal constant $l=l(n,b,c,p,\gamma)>0$ and $t$ close enough to $T^\ast$.

Now, we split the proof of \eqref{f78} into the following three steps.

\textbf{Step 1.} We claim that the rescaled solution $v(\tau_0)$ satisfies all the conditions of Propositions \ref{pro4.1} and \ref{pro4.2} for any $\tau_0\in[0, e^{N(t)}]$.

In fact, with $\tau_\ast=\frac{2}{\lambda_u(t)}$, since $\tau_\ast\leq\frac{t}{\lambda^2_u(t)}$ due to \eqref{f77}, $v(\tau)$ is the solution to \eqref{a0} on $[0,\tau_\ast]$.
By \eqref{f79}, we have
\[\tau_\ast^{1-\frac{2s_c}{2-b}}|E(v(0))|=(2\lambda_u(t))^{1-\frac{2s_c}{2-b}}|E(u_0)|\rightarrow0 \quad\mbox{as}\ \  t\rightarrow T^\ast,\]
thus \eqref{f20} follows. Recalling the estimate \eqref{f80} and using \eqref{f79}, we get
\begin{equation}\nonumber
E(v(0))\geq \frac{1}{2}\left(1-\frac{\|v(0)\|^p_{\gamma, \sigma_0}}{\|\mathcal{Q}\|^p_{\gamma, \sigma_0}}\right),
\end{equation}
which implies \eqref{f21}. Thus, for the universal constant $\alpha_2$ given in \eqref{f23}, we write
\begin{equation}\label{f81}
M(t)=\frac{4\|v(0)\|_{\gamma, \sigma_0}}{\|\mathcal{\mathcal{Q}}\|_{\gamma, \sigma_0}}\geq2\quad\mbox{and}\quad L(t)=(100M(t)^{\alpha_2})^{\frac{2-b}{2(2-b-2s_c)}},
\end{equation}
which will be needed in the next step.

On the other hand, for all $\tau_0\in[0, e^{N(t)}]$, we deduce $\tau_0\in[0, \frac{\tau_\ast}{2}]$. Moreover, \eqref{f84} implies that
\[\lambda_v(\tau_0)=(\lambda_u(t)\|\nabla u(t-\lambda_u^2(t)\tau_0)\|_{b,2})^{-\frac{2-b}{2-b-2s_c}}=\frac{\lambda_u(t-\lambda_u^2(t)\tau_0)}{\lambda_u(t)},\]
where $t-\lambda_u^2(t)\tau_0\geq t-\lambda_u(t)\rightarrow T^\ast$, as $t\rightarrow T^\ast$. This together with \eqref{f79} yields
\[\lambda_v^{2-\frac{4s_c}{2-b}}(\tau_0)|E(v(0))|=\lambda_u^{2-\frac{4s_c}{2-b}}(t-\lambda_u^2(t)\tau_0)|E(u_0)|\rightarrow0\quad  \mbox{as}\ \ t\rightarrow T^\ast.\]
So, $v(\tau_0)$ also satisfies the conditions of Proposition \ref{pro4.2}.

\textbf{Step 2.} By the definition of $L(t)$, we consider two cases:

If $L(t)\geq e^{\frac{\sqrt{N(t)}}{2}}$, it follows from \eqref{f81} that
\[\|v(0)\|_{\gamma, \sigma_0}=CL(t)^{\frac{2(2-b-2s_c)}{(2-b)\alpha_2}}\geq Ce^{\frac{(2-b-2s_c)\sqrt{N(t)}}{(2-b)\alpha_2}}\geq\sqrt{N(t)},\]
then \eqref{f78} is proved.

If $L(t)<e^{\frac{\sqrt{N(t)}}{2}}$. To our aim \eqref{f78}, we need to prove that there exists $\tau_i\in[0, e^{i}]$ such that
\begin{equation}\label{f82}
F(\tau_i)\leq L(t)^{\frac{2}{2-b}}\quad\mbox{and}\quad \frac{e^{\frac{i-1}{2}}}{10L(t)}\leq\lambda_v(\tau_i)\leq \frac{10e^{\frac{i}{2}}}{L(t)}
\end{equation}
for all $i\in[\sqrt{N(t)}, N(t)]$, where $F(\tau)=\frac{\tau^{\frac{1}{2-b}}}{\lambda_v^{\frac{2}{2-b}}(\tau)}$ is defined as in \eqref{f43}.

We postpone the argument of \eqref{f82} to the third step. Assume \eqref{f82} holds true. Then there exists a universal constant $\alpha_4=\alpha_4(b, s_c, \alpha_2)>0$ such that
\[D_\ast=M(t)^{\alpha_3}\max\{1, F(\tau_i)^{\frac{2-b+2s_c}{2-b-2s_c}}\}\leq M(t)^{\alpha_4},\]
so we apply \eqref{f44} of Proposition \ref{pro4.2} to $v(\tau_i)$ and get
\begin{equation}\label{f89}
\frac{1}{\lambda_v^{\frac{4s_c}{2-b}}(\tau_i)}\int_{|x|\leq M(t)^{\alpha_4}\lambda_v^{\frac{2}{2-b}}(\tau_i)}|v(0)|^2dx\geq C_2, \quad\quad \forall\ i\in[\sqrt{N(t)}, N(t)].
\end{equation}
Now by defining the annulus
\begin{equation}\nonumber
\mathcal{C}_i=\{x:\quad \frac{\lambda_v^{\frac{2}{2-b}}(\tau_i)}{M(t)^{\alpha_4}}\leq|x|\leq M(t)^{\alpha_4}\lambda_v^{\frac{2}{2-b}}(\tau_i)\},
\end{equation}
we claim that there exists a constant $C_4=C_4(\sigma_0)>0$ such that
\begin{equation}\label{f86}
\int_{\mathcal{C}_i}|x|^\gamma|v(0)|^{\sigma_0}dx\geq\frac{C_4}{M(t)^{\alpha_4s_c\sigma_0}},\quad\quad \forall\ i\in[\sqrt{N(t)}, N(t)].
\end{equation}

To do it, we use the H\"{o}lder inequality to obtain
\[\int_{|x|\leq \frac{\lambda_v^{\frac{2}{2-b}}(\tau_i)}{M(t)^{\alpha_4}}}|v(0)|^2dx\leq \frac{C\lambda_v^{\frac{4s_c}{2-b}}(\tau_i)}{M(t)^{2s_c\alpha_4}}\|v(0)\|^2_{\gamma, \sigma_0},\]
which together with \eqref{f89} yields
\begin{eqnarray}\nonumber
\lefteqn{\int_{\mathcal{C}_i}|v(0)|^2dx=\int_{|x|\leq M(t)^{\alpha_4}\lambda_v^{\frac{2}{2-b}}(\tau_i)}|v(0)|^2dx-\int_{|x|\leq \frac{\lambda_v^{\frac{2}{2-b}}(\tau_i)}{M(t)^{\alpha_4}}}|v(0)|^2dx}\\\label{f85}
&&\quad\quad\quad\quad\geq\lambda_v^{\frac{4s_c}{2-b}}(\tau_i)\left(C_2-\frac{CM(t)^2}{M(t)^{2s_c\alpha_4}}\right)\geq\lambda_v^{\frac{4s_c}{2-b}}(\tau_i)\frac{C_2}{2}
\end{eqnarray}
for $\alpha_4>0$ large enough.  Applying the H\"{o}lder inequality again to the left hand side of \eqref{f85}, we obtain
\[\int_{\mathcal{C}_i}|v(0)|^2dx\leq C\lambda_v^{\frac{4s_c}{2-b}}(\tau_i)M(t)^{2s_c\alpha_4}(\int_{\mathcal{C}_i}|x|^\gamma|v(0)|^{\sigma_0}dx)^{\frac{2}{\sigma_0}}. \]
Therefore, combining this last inequality with \eqref{f85}, we conclude \eqref{f86}.

To this end, let $p(t)>1$ be an integer such that
\begin{equation}\label{f88}
10^{\frac{6}{2-b}}M(t)^{2\alpha_4}\leq e^{\frac{p(t)}{2-b}}\leq10^{\frac{16}{2-b}}M(t)^{2\alpha_4}.
\end{equation}

If $p(t)<\sqrt{N(t)}$, then for any $(i, i+p(t))\in [\sqrt{N(t)}, N(t)]\times [\sqrt{N(t)}, N(t)]$, we deduce from \eqref{f82} and \eqref{f88} that
\[\lambda_v^{\frac{2}{2-b}}(\tau_{i+p(t)})\geq\frac{e^{\frac{i+p(t)-1}{2-b}}}{(10L(t))^{\frac{2}{2-b}}}\geq M(t)^{2\alpha_4}\frac{10^{\frac{4}{2-b}}e^{\frac{i-1}{2-b}}}{L(t)^{\frac{2}{2-b}}}\geq M(t)^{2\alpha_4}\lambda_v^{\frac{2}{2-b}}(\tau_{i}).\]
This shows that $\mathcal{C}_i\cap\mathcal{C}_{i+p(t)}=\emptyset$. Therefore, there are at least $\frac{N(t)}{10p(t)}>\frac{\sqrt{N(t)}}{10}$ disjoint annuli satisfying the uniform lower bound \eqref{f86}. Thus we sum over these sets and obtain
\[\|v(0)\|_{\gamma, \sigma_0}\geq\sum_{k=0}^{\frac{N(t)}{10p(t)}}\int_{\mathcal{C}_{\sqrt{N(t)}+kp(t)}}|x|^\gamma|v(0)|^{\sigma_0}dx\geq\frac{C_4\sqrt{N(t)}}{10M(t)^{\alpha_4s_c\sigma_0}}=\frac{C_4\sqrt{N(t)}}{\|v(0)\|^{\alpha_4s_c\sigma_0}_{\gamma, \sigma_0}}.\]
Otherwise, if $p(t)>\sqrt{N(t)}$, then by \eqref{f81},
\[\|v(0)\|^{2\alpha_4}_{\gamma, \sigma_0}=CM(t)^{2\alpha_4}\geq Ce^{\frac{p(t)}{2-b}}\geq Cp(t)\geq C\sqrt{N(t)},\]
which concludes the proof of \eqref{f78}.

\textbf{Step 3.} We are left with proving \eqref{f82}. We consider the set
\[\mathcal{S}=\{\tau\in[0, e^{i}]:\ \lambda_v(\tau)>\frac{e^{\frac{i+1}{2}}}{L(t)}\}.\]
Noting from \eqref{f79} that
\[\lambda_v(0)=1<\frac{e^{\frac{\sqrt{N(t)}}{2}}}{L(t)}\leq\frac{e^{\frac{i+1}{2}}}{L(t)}.\]
Thus $0\notin\mathcal{S}$.
If $\mathcal{S}$ is nonempty, the continuity of $\lambda_v(\tau)$ implies that there exists a $\tau_i\in[0, e^{i}]$ such that $\lambda_v(\tau_i)=\frac{e^{\frac{i+1}{2}}}{L(t)}\in \left[\frac{e^{\frac{i-1}{2}}}{10L(t)}, \frac{10e^{\frac{i}{2}}}{L(t)}\right]$. Furthermore,
\[F(\tau_i)=\left(\frac{\tau_i}{\lambda_v^2(\tau_i)}\right)^{\frac{1}{2-b}}\leq L(t)^{\frac{2}{2-b}},\]
which arrives at \eqref{f82}. If $\mathcal{S}$ is empty, then for all $\tau\in[0, e^{i}]$, we have
\[\lambda_v(\tau)\leq\frac{e^{\frac{i+1}{2}}}{L(t)}\leq\frac{10e^{\frac{i}{2}}}{L(t)}.\]

To show the bound of $F(\tau_i)$ in \eqref{f82}, we assume by contradiction that
\[F(\varsigma)\geq L(t)^{\frac{2}{2-b}},\quad\ \mbox{i.e.}\quad\   \|\nabla u(\varsigma)\|_{b,2}\geq \left(\frac{L(t)^{2}}{\varsigma}\right)^{\frac{2-b-2s_c}{2(2-b)}} \]
for all $\varsigma\in [e^{i-1}, e^i]$. By Proposition \ref{pro4.1}, we have
\begin{eqnarray}\nonumber
\lefteqn{M_0^{\alpha_2}e^{i(1+\frac{2 s_c}{2-b})}\geq\int_0^{e^i}(e^i-\varsigma)\|\nabla u(\varsigma)\|_{b,2}^2d\varsigma\geq\int_{e^{i-1}}^{\frac{e^i}{2}}\varsigma\|\nabla u(\varsigma)\|_{b,2}^2d\varsigma}\\\nonumber
&&\quad\quad\quad\quad\geq L(t)^{\frac{2(2-b-2s_c)}{2-b}}\int_{e^{i-1}}^{\frac{e^i}{2}}\frac{\varsigma}{\varsigma^{\frac{2-b-2s_c}{2-b}}}d\varsigma\\\nonumber
&&\quad\quad\quad\quad>\frac{L(t)^{\frac{2(2-b-2s_c)}{2-b}}}{100}e^{i(1+\frac{2s_c}{2-b})}=M(t)^{\alpha_2}e^{i(1+\frac{2 s_c}{2-b})},\quad\quad\quad
\end{eqnarray}
which is a contradiction, where we have used the definition of $L(t)$ in \eqref{f81}. Therefore, there exists a $\tau_i\in[e^{i-1}, e^{i}]$ such that $F(\tau_i)\leq L(t)^{\frac{2}{2-b}}$. We further have
\[\lambda_v(\tau_i)=\frac{\sqrt{\tau_i}}{F(\tau_i)^{\frac{2-b}{2}}}\geq\frac{e^{\frac{i-1}{2}}}{10L(t)}.\]
So, \eqref{f82} is proved.
We complete the proof of Theorem \ref{thm3}.

\subsection*{Acknowledgments}
This work is supported by the National Natural Science Foundation of China (Grant No. 12401147), and the Nature Science Foundation of Zhejiang Province (Grant No. LQ21A010011).

\subsection*{Declarations}
\textbf{Competing Interests}
The authors have no conflicts to disclose.

\end{document}